\documentclass[12pt, a4paper]{amsart}

\usepackage{microtype} 
\usepackage[textwidth=391pt, marginparwidth=70pt, centering]{geometry}

\usepackage{graphicx}
\usepackage{xcolor}

\usepackage{amssymb}
\usepackage{mathtools}
\usepackage{stmaryrd} 
\usepackage{tikz-cd}
\usepackage{aliascnt}	

\newtheorem*{thm*}{Theorem}

\newtheorem{thm}{Theorem}[section]

\newaliascnt{lemma}{thm}
\newtheorem{lemma}[lemma]{Lemma}
\aliascntresetthe{lemma}

\newaliascnt{prop}{thm}
\newtheorem{prop}[prop]{Proposition}
\aliascntresetthe{prop}

\newaliascnt{corollary}{thm}
\newtheorem{corollary}[corollary]{Corollary}
\aliascntresetthe{corollary}

\newaliascnt{conj}{thm}
\newtheorem{conj}[conj]{Conjecture}
\aliascntresetthe{conj}

\newaliascnt{remark}{thm}
\theoremstyle{remark}
\newtheorem{remark}[remark]{Remark}
\aliascntresetthe{remark}

\theoremstyle{definition}
\newtheorem*{ack}{Acknowledgements}

\newcommand{\NN}{\mathbb{N}}
\newcommand{\ZZ}{\mathbb{Z}}
\newcommand{\QQ}{\mathbb{Q}}
\newcommand{\RR}{\mathbb{R}}
\newcommand{\CC}{\mathbb{C}}
\newcommand{\QQab}{\QQ^{\mathrm{ab}}}

\DeclareMathOperator{\Gal}{Gal}

\newcommand*{\artin}[3]{\left(#3, #1 / #2 \right)}
\renewcommand{\P}{\mathfrak{P}}
\newcommand{\p}{\mathfrak{p}}
\newcommand{\q}{\mathfrak{q}}

\usepackage[
	backend = biber,
	style = alphabetic,
	backref = true,
	url = false,
	doi = true,
	isbn = false,
]{biblatex}
\DefineBibliographyStrings{english}{
  backrefpage  = {$\uparrow$\hspace{-2.5pt}},
  backrefpages = {$\uparrow$\hspace{-2.5pt}},}

\DeclareFieldFormat[article]{title}{#1}

\DeclareFieldFormat[
    book,
    incollection,
    inproceedings,
    misc
]{title}{\mkbibemph{#1}}

\renewbibmacro{in:}{}

\definecolor{unibas-red}{RGB}{210, 5, 55}
\definecolor{unibas-mint}{RGB}{165, 215, 210}
\definecolor{unibas-mint-dark}{RGB}{0, 110, 110}
\definecolor{unibas-mint-dark-HEX}{HTML}{1F6F78}
\definecolor{unibas-hantr-hell}{RGB}{119, 119, 119}

\usepackage[
	colorlinks, 
	citecolor=unibas-red, 
	linkcolor=unibas-red,
	urlcolor=., 
	linktocpage
]{hyperref} 

\usepackage[nameinlink]{cleveref}
\crefname{thm}{theorem}{theorems}
\Crefname{thm}{Theorem}{Theorems}

\crefname{lemma}{lemma}{lemmas}
\Crefname{lemma}{Lemma}{Lemmas}

\crefname{prop}{proposition}{propositions}
\Crefname{prop}{Proposition}{Propositions}

\crefname{corollary}{corollary}{corollaries}
\Crefname{corollary}{Corollary}{Corollaries}

\crefname{conj}{conjecture}{conjectures}
\Crefname{conj}{Conjecture}{Conjectures}

\crefname{remark}{remark}{remarks}
\Crefname{remark}{Remark}{Remarks}

\title[A lower bound for the canonical height]{A lower bound for the 
Call--Silverman height in cyclotomic extensions}
\date{\today}

\author[Alessio~Cangini]{Alessio~Cangini}
\address{
	University of Basel, 
	Department of Mathematics and Computer Science, 
	Spiegelgasse~$1$, 
	$4051$ Basel,
	Switzerland}
\email{alessio.cangini@unibas.ch}
\urladdr{https://sites.google.com/view/alessio-cangini/}

\thanks{The author acknowledges support from the Swiss National 
Science Foundation Grant ``Rational points, arithmetic dynamics, and 
heights'' $\mathrm{n^\circ} 200020\_219397$.}

\begin{document}

\begin{abstract}
For quadratic postcritically finite polynomials defined over a number 
field $K$, we establish lower bounds for the Call--Silverman height 
of wandering algebraic integers lying in cyclotomic extensions of $K$ 
and provide partial results for the nonintegral case. Our proof 
combines potential theory and algebraic number theory with 
equidistribution techniques.
\end{abstract}

\maketitle
\section{Introduction}

In this paper, for postcritically finite quadratic polynomials of the 
form $T^2 + c$, we prove a lower bound for the canonical height of 
wandering algebraic integers in cyclotomic extensions. Moreover, we 
establish some partial results for the nonintegral~case.

Let $K$ be a number field and fix an algebraic closure $\overline K$. 
A \textit{dynamical system defined over} $K$ is a polynomial $f \in K[T]$ 
of degree at least $2$. Set $f^0(T) = T$ and denote the $n$-fold 
iterate of $f$ by $f^n$ for all $n \in \NN = \{1, 2, \dots\}$. An 
element $\alpha \in \overline K$ is \textit{preperiodic} if its 
forward orbit $\{f^n(\alpha) \colon n \ge 0\}$ is finite and 
\textit{wandering} otherwise.

Given a dynamical system $f$ defined over $K$, the \textit{canonical 
height} introduced by Call and Silverman~\cite{MR1255693} is the 
unique function $\hat h_f \colon \overline K \rightarrow \RR_{\ge 0}$ 
satisfying
\begin{itemize}
    \item[(i.)]  $\hat h_f(f(\alpha)) = \deg f \cdot \hat h_f(\alpha)$;
    \item[(ii.)] $\hat h_f(\alpha) = h(\alpha) + O(1)$
\end{itemize}
where $h$ denotes the (absolute logarithmic) Weil height. If $f(T) = T^2$, 
then the canonical height coincides with the Weil height. Properties 
(i.) and (ii.) together with Northcott's theorem imply that $\hat h_f(\alpha) = 0$
if and only if $\alpha$ is preperiodic. This fact is often referred 
to as the Dynamical Kronecker Theorem, see~~\cite[Theorem~3.22]{MR2316407}.

We focus on postcritically finite quadratic dynamical systems of the 
form $f(T) = T^2 + c \in K[T]$ and on wandering points in the 
cyclotomic closure of $K$. We say that $f(T) = T^2 + c$ is 
\textit{postcritically finite} (PCF) if $0$ is preperiodic.
The \textit{cyclotomic closure} of $K$ in $\overline{K}$ is the 
compositum 
\[
    K^\mathrm{cyc} = \bigcup_{n \ge 1} K(\zeta_n)
\]
where $\zeta_n$ denotes an $n$th root of unity. It is an infinite 
abelian extension of~$K$. If $K = \QQ$, then $\QQ^\mathrm{cyc}$ 
coincides with the maximal abelian extension of $\QQ$ by the 
Kronecker--Weber theorem~\cite[see Ch. VI,
Corollary~6.3 and Proposition~6.7]{MR1697859}. 

Our main result is the following lower bound for the height of 
wandering algebraic integers in $K^\mathrm{cyc}$.
\begin{thm}\label{thm:main-thm}
If $K$ is a number field unramified above $2$ and $f(T) = T^2 + c \in K[T]$ 
is PCF, then there exists $C = C(K, f) > 0$ such that $\hat h_f(\alpha) 
\ge C$ for all wandering algebraic integers $\alpha \in K^\mathrm{cyc}$.  
\end{thm}
In other words, every algebraic integer $\alpha \in K^\mathrm{cyc}$ 
satisfying $\hat h_f(\alpha) < C$ is preperiodic.

There exist examples of PCF quadratic polynomials defined over a 
number field unramified above $2$. For instance, if $0$ is periodic 
for $f(T) = T^2 + c$, i.e.\ if there exists $n \in \NN$ such that $f^n(0) 
= 0$, then $2$ is unramified in $\QQ(c)$, see~\cite[Exposé~XIX, Lemme~2]{MR762431}. 
We thus highlight the following corollaries.
\begin{corollary}\label{cor:0-per}
Let $f(T) = T^2 + c \in \overline\QQ [T]$. If $0$ is periodic, then 
there exists $C = C(f) > 0$ such that $\hat h_f(\alpha) \ge C$ for 
all wandering algebraic integers $\alpha \in {\QQ(c)}^\mathrm{cyc}$.
\end{corollary}

\begin{corollary}\label{cor:ab-ext}
If $f(T) \in \{ T^2 - 2, T^2 - 1, T^2\}$, then there exists $C = C(f) 
> 0$ such that $\hat h_f(\alpha) \ge C$ for all wandering algebraic 
integers $\alpha \in \QQab$.
\end{corollary}

It will follow from the proof of~\Cref{thm:main-thm} that in the 
unramified setting, our result extends to a lower bound for all 
algebraic elements. Let $K$ be a number field unramified above $2$. 
The field 
\[
    K^{\mathrm{cyc, unr\ 2}} = \bigcup_{n \text{ odd}} K(\zeta_n)
\]
is the maximal subfield of $K^\mathrm{cyc}$ which is unramified 
above~$2$.
\begin{thm}\label{thm:main-thm-unr}
If $K$ is a number field unramified above $2$ and $f(T) = T^2 + c \in 
K[T]$ is PCF, then there exists $C = C(K, f) > 0$ such that 
$\hat h_f(\alpha) \ge C$ for all wandering $\alpha \in K^\mathrm{cyc, unr\ 2}$.  
\end{thm}

For the full cyclotomic closure, we obtain a result for the elements 
which are $2$-adically integral for at least one embedding.

\begin{thm}\label{thm:main-thm-ram}
Let $K$ be a number field unramified above $2$. Let $f(T) = T^2 + c 
\in K[T]$ be PCF.\ There exists $C = C(K, f) > 0$ such that 
$\hat h_f(\alpha) \ge C$ for all wandering $\alpha \in K^\mathrm{cyc}$ 
for which there exists an embedding $\sigma \colon \QQ(\alpha) 
\rightarrow \overline{\QQ}_2$ such that $\sigma(\alpha)$ is 
$2$-adically integral.
\end{thm}

If $f(T) = T^2$,~\Cref{cor:ab-ext} recovers results of Amoroso, 
Dvornicich and Zannier~\cite{MR1740514, MR1817715, MR2651944} in the 
setting of algebraic integers. The case $f(T) = T^2 - 2$ is closely 
related to the first one since $\hat h_{T^2 - 2}(\alpha + 
\alpha^{-1}) = 2 h(\alpha)$ for all $\alpha \in \overline{\QQ}
\setminus\{0\}$, see~\cite[Proposition~5.1]{MR3129749}. However, note 
that it is not covered by~\Cref{cor:0-per} since $0$ is not periodic. 
The case $f(T) = T^2 - 1$ is~new.

\subsection{Context and motivation: points of small canonical height}

By the Dynamical Kronecker Theorem, the preperiodic points are 
precisely the points of canonical height zero. In general, wandering 
algebraic numbers can have arbitrarily small canonical height. 
Indeed, let $\alpha \in \overline{K}$ be wandering, and choose a 
sequence $(\alpha_n)_{n \ge 0}$ of algebraic numbers such that 
$\alpha_0 = \alpha$ and $f(\alpha_n) = \alpha_{n - 1}$ for every $n 
\ge 1$. Then $\alpha_n$ is wandering for all $n \ge 0$. Since 
$f^n(\alpha_n) = \alpha$, by property (i.) of the canonical height, 
it follows that $\hat h_f(\alpha_n) = (\deg f)^{-n}\hat h_f(\alpha)$. 
So $\hat h_f(\alpha_n)$ tends to $0$ as $n$ tends to infinity. In 
particular, property (ii.) implies that the Weil heights $h(\alpha_n)$ 
are uniformly bounded. Moreover, the above formula implies that the 
values $\hat h_f(\alpha_n)$ are pairwise distinct, and hence so are 
the $\alpha_n$. Northcott's theorem therefore implies that 
$[K(\alpha_n) \colon K]$ tends to infinity. Thus, although the 
canonical height of wandering algebraic numbers can be arbitrarily 
small, this necessarily occurs along points of increasing degree. It 
is natural to ask whether the canonical height can be bounded from 
below in terms of the degree. The Dynamical Lehmer Conjecture 
predicts such a bound.
\begin{conj}[Dynamical Lehmer Conjecture]\label{conj:dyn-l}
Let $K$ be a number field and let $f \in K[T]$ be a dynamical system. 
Then, there exists $C = C(K, f) > 0$ such that $\hat h_f(\alpha) \ge 
C/[K(\alpha) \colon K]$ for all wandering $\alpha \in \overline K$.
\end{conj}

For a more general statement, see~\cite[Conjecture 3.25]{MR2316407}.
The Dynamical Lehmer Conjecture is open already in the case $K = \QQ$ 
and $f(T) = T^2$. The strongest general lower bound currently known 
in this case is due to Dobrowolski~\cite{MR543210}. 

In the literature, there are two main strategies to obtain evidence 
for~\Cref{conj:dyn-l}. The first one is to restrict to subsets 
of~$\overline K$ or to specific dynamical systems. This strategy 
usually leads to stronger bounds than the one predicted by~\Cref{conj:dyn-l}. 
In the case $f(T) = T^2$, it has produced substantial evidence. In 
this direction, Bombieri and Zannier~\cite{MR1898444} defined a 
subset of~$\overline K$ to have the \textit{Bogomolov property} 
relative to the Weil height if 
$0$ is isolated in the set of Weil heights attained by its elements.
Fields known to have the Bogomolov property include the maximal 
abelian extension of a number field and its finite 
extensions~\cite{MR1740514, MR1817715, MR2651944}; the field of 
totally real numbers~\cite{MR360515}; the field of totally $p$-adic 
numbers~\cite{MR1898444}; and the field generated over $\QQ$ by the 
coordinates of all torsion points of a given elliptic curve over 
$\QQ$~\cite{MR3090783}. For surveys on Lehmer-type problems, 
see~\cite{MR2604658, MR1990219}.

There are also examples for more general dynamical systems. Pottmeyer 
\cite{MR3129749} classified the polynomials $f \in \overline \QQ[T]$ 
for which the field of totally real numbers has the Bogomolov 
property relative to $\hat h_f$. 
Looper~\cite{looper2021bogomolovpropertycanonicalheight} established 
Bogomolov-type results for polynomials in $K[T]$ with a finite 
superattracting periodic point and a nonarchimedean place of bad 
reduction, for points lying in the maximal abelian extension of $K$.
Plessis and Sahoo~\cite{plessis2025lowerboundsheightsalgebraic} 
derived lower bounds for points in the maximal unramified extension 
of $K$ at a fixed place.
From this perspective,~\Cref{thm:main-thm} establishes an integral 
version of the Bogomolov property relative to the canonical height 
for $K^{\mathrm{cyc}}$.

The second strategy consists of replacing the canonical height with 
more tractable quantities. It originates from the study of the 
Schinzel--Zassenhaus Conjecture~\cite{MR0175882}, which predicts the 
existence of a lower bound away from $1$ for the \textit{house} of a 
nonzero algebraic integer which is not a root of unity, where the 
house is the maximum absolute value among its Galois conjugates. The 
Schinzel--Zassenhaus Conjecture was proved by Dimitrov~\cite{dimitrov}. 
Habegger and Schmidt~\cite{MR4726503} later adapted Dimitrov's ideas 
to arithmetic dynamics, introducing a dynamical analogue of the house 
and using it to obtain a lower bound for the canonical height of 
wandering algebraic integers for certain unicritical polynomials of 
prime degree. Their resulting bound decays as the inverse square of 
the algebraic degree. More recently, they extended their method to a 
broader class of polynomials; 
see~\cite{habegger2026dynamicalcanonicalheightsfinite}.

\subsection{Proof strategy}
Our proof combines Dimitrov's approach to the Schinzel--Zassenhaus 
Conjecture~\cite{dimitrov} and its adaptation to dynamics due to 
Habegger--Schmidt~\cite{MR4726503} with equidistribution results for 
points of small canonical height.

Let $K$ be a number field and fix a PCF $f(T) = T^2 + c \in K[T]$. 
Let $\zeta$ be a root of unity and let $\alpha \in K(\zeta)$ be 
wandering. The basic idea is to consider the square root $\Phi$ of 
$(T - f^k(\alpha)) / (T - \sigma(f^h(\alpha)))$ where $\sigma \in 
\Gal(K(\zeta) / K)$ and $(k, h) \in \NN^2$ are chosen so that $\Phi$ 
is integral for all places of $K(\zeta)$ outside of a finite set $S$ 
of nonarchimedean places. To construct such a~$\sigma$, we 
distinguish according to whether the prime $2$ ramifies in $K(\zeta)$.
At the places in $S$, the power series $\Phi$ converges on the 
complement of a suitable disk. At the archimedean places, the 
power series extends holomorphically to the complement of the line 
segment $[f^k(\alpha), \sigma(f^h(\alpha))]$. The transfinite 
diameter of disks and segments can be computed explicitly at each 
place. If $\Phi$ is nonrational, the aforementioned integrality 
allows us to apply the P\'olya--Bertrandias theorem and combine the 
local information about the transfinite diameters into a global 
inequality, yielding a first lower bound for the Weil height of 
certain iterates of $\alpha$. 
To conclude the argument, we invoke the Baker--Hsia equidistribution 
theorem~\cite{MR2164622} for points of small canonical height, as 
well as Pritsker's result~\cite[Theorem~1.4]{MR4651642} on the Mahler 
measure of the iterates of a polynomial. These results allow us to 
relate the Weil height of an iterate to its canonical height with 
suitable precision. Crucially, our method relies on the fact that the 
constant in~\cite[Theorem~1.4]{MR4651642} is less than $\log(2)/2$.

The preceding argument applies provided that $\Phi$ is nonrational, 
i.e.\ if $\sigma(f^h(\alpha)) \neq f^k(\alpha)$. This issue behaves 
differently according to the ramification of $2$ in $K(\zeta)$. While 
it can be easily excluded in the unramified case, the ramified case 
is more delicate. In that case, we choose $(k, h) = (1, 1)$ and 
$\sigma \in \Gal(K(\zeta) / K(\zeta^2))$ to be the generator. If 
$f(\alpha) \notin K(\zeta^2)$, then $\sigma (f(\alpha))\neq f(\alpha)$, 
so the preceding argument applies. The remaining case is $f(\alpha) 
\in K(\zeta^2)$. The key point is to show that if $2$ is ramified in 
$K(\zeta^4)$, the number $\alpha$ is integral with respect to a prime 
ideal $\p$ above $2$, $\alpha \in K(\zeta) \setminus K(\zeta^2)$ and 
$f(\alpha) \in K(\zeta^2)$, then $f^2(\alpha) \notin K(\zeta^4)$, 
allowing to reduce to a setting in which the argument works. The 
proof of this last property relies on the analysis of the $2$-adic 
valuation of the parameter $c$.

\subsection{Outline}

In~\Cref{section:congruences} we define the Galois automorphism  
$\sigma$ and lay the groundwork to define $\Phi$. Moreover, we study 
the $2$-adic valuation of the parameter $c$ and prove the 
nondegeneracy property discussed above. 
In~\Cref{section:dimitrov}, after providing the necessary background 
on transfinite diameters, we define the function $\Phi$. At the end, 
we recall the P\'olya--Bertrandias theorem. 
In~\Cref{section:equidistribution}, after a brief review of the main 
equidistribution theorems in arithmetic dynamics, we establish a 
height inequality for algebraic numbers of small canonical height 
that tightly relates their Weil and canonical heights. Finally, we 
prove~\Cref{thm:main-thm,thm:main-thm-unr,thm:main-thm-ram} 
in~\Cref{section:proof}.
\begin{ack}
I am very grateful to my PhD advisor Philipp Habegger for suggesting
this topic, for the many insightful discussions, and for the comments 
on the first draft of this paper. I thank Hang Fu for his valuable 
suggestions regarding the study of the $2$-adic valuation of the PCF 
parameters.
\end{ack}

\section{Congruences}\label{section:congruences}

Let $K/\QQ$ be a finite Galois extension of degree $d = [K \colon \QQ]$. 
Denote the ring of integers of $K$ by $O_K$. Let $\p$ be a prime 
ideal of $K$ lying above~$2$. The quotient $O_K / \p$ is a finite 
field extension of $\mathbb{F}_2$. The inertia degree of $\p$ is 
defined as $r = r(\p | 2) = [O_K / \p \colon \mathbb{F}_2]$. The 
ramification index $e(\p | 2)$ is defined as the exponent of $\p$ in 
the prime ideal factorization of $2O_K$. We assume that~$2$ is 
unramified in $K$. In particular, $e(\p | 2) = 1$. All extensions of 
the $2$-adic valuation are normalized such that the valuation of $2$ 
is $1$.

\subsection{Auxiliary lemmas}

Before analyzing the congruences, we state the following auxiliary 
lemmas. The first one is a consequence of the Strong Approximation 
Theorem~\cite[Ch. II, Section 15, p. 67]{alma9928223550105504}. Given 
a number field $L$, we denote by $M_L$ the set of its places and by 
$M_L^0$ the set of its nonarchimedean places. 
\begin{lemma}\label{lemma:ad-denominator}
Let $L$ be a number field and fix a prime integer $p$. For each 
$\alpha \in L$ there exists $\beta \in O_L$ such that $\alpha \beta 
\in O_L$ and $|\beta|_v = {\max\{1, |\alpha|_v\}}^{-1}$ for all 
places $v \in M_L$ above $p$.
\end{lemma}
\begin{proof}
The case $\alpha = 0$ is trivial; thus we assume $\alpha \neq 0$. 
Define $\Sigma_0 = \{w \in M_L^0 \colon |\alpha|_w > 1\}$ and $\Sigma 
= \Sigma_0 \cup \{w \in M_L^0 \colon w | p\}$. Fix a nonarchimedean 
place $w_0 \in M_L$. The Strong Approximation Theorem implies that 
there exists $\beta \in L$ such that
\[
    \begin{cases}
    |\beta - \alpha^{-1}|_w < {\max\{1, |\alpha|_w\}}^{-1}    &\text{ if } w \in \Sigma_0,\\
    |\beta - 1|_w < {\max\{1, |\alpha|_w\}}^{-1}              &\text{ if } w \in \Sigma \setminus \Sigma_0,\\
    |\beta|_w \le 1                                           &\text{ if } w \in M_L \setminus \left(\Sigma \cup \{w_0\}\right).
    \end{cases}
\]
If $w \in \Sigma_0$, then $|\alpha|_w > 1$, from which $|\beta - 
\alpha^{-1}|_w < |\alpha|_w^{-1} < 1$. Therefore, an application of 
the ultrametric inequality implies that
\[
    |\beta|_w
    = |\beta - \alpha^{-1} + \alpha^{-1}|_w
    = |\alpha^{-1}|_w
    = {\max\{1, |\alpha|_w\}}^{-1}.
\]
If $w \in \Sigma \setminus \Sigma_0$, then $|\alpha|_w \le 1$, from 
which $|\beta - 1|_w < 1$. Therefore, the ultrametric inequality 
again yields
\[
    |\beta|_w
    = |\beta - 1 + 1|_w
    = 1
    = {\max\{1, |\alpha|_w\}}^{-1}.
\]
We conclude that $|\beta|_w = {\max\{1, |\alpha|_w\}}^{-1}$ for all 
$w \in \Sigma$. Hence, the equality holds true for all $w \in M_L$ 
above $p$ since they are contained in $\Sigma$.

It remains to prove that $\beta$ and $\alpha\beta$ are algebraic 
integers. For every nonarchimedean place $w$, we check that $|\beta|_w 
\le 1$. Indeed, if $w \in \Sigma$, then 
\(
    |\beta|_w 
    = {\max\{1, |\alpha|_w\}}^{-1} \le 1
\)
while if $w \notin \Sigma \cup \{w_0\}$, the construction explicitly 
gives $|\beta|_w \le 1$. Hence $\beta \in O_L$.
Next, we show that $|\alpha \beta|_w \le 1$. If $w \in \Sigma$, then
\[
    |\beta|_w
    = {\max\{1, |\alpha|_w\}}^{-1}
    \le |\alpha|_w^{-1},
\]
meaning $|\alpha \beta|_w \le 1$. If $w \notin \Sigma$, then $w 
\notin \Sigma_0$, and hence $|\alpha|_w \le 1$. Since by construction 
we have $|\beta|_w \le 1$, we have $|\alpha\beta|_w = |\alpha|_w |
\beta|_w \le 1$. Therefore $\alpha \beta \in O_L$.
\end{proof}

The following lemma will be helpful in determining the congruences 
modulo~$4$.
\begin{lemma}\label{lemma:iterates-of-f}
Let $K$ be a number field. Let $f(T) = T^2 + c \in O_K[T]$. For all 
$k \ge 1$, we have
\begin{equation*}\label{eqn:poly-iterates-of-f}
    f^k(T) 
    \equiv {\left(T^{2^{k - 1}} + f^{k-1}(0)\right)}^2 + c \mod 4 O_K[T].
\end{equation*}
Moreover, if $L / K$ is a finite extension and $\alpha \in L$, let 
$\beta \in O_L$ be such that $\beta\alpha \in O_L$. Then 
\begin{equation}\label{eqn:iterates-of-f}
    \beta^{2^k}f^k(\alpha) 
    \equiv \beta^{2^k}{\left(\alpha^{2^{k - 1}} + f^{k - 1}(0)\right)}^2 + \beta^{2^k}c 
    \mod 4 O_L
\end{equation}
for all $k \ge 1$.
\end{lemma}
\begin{proof}
The first part of the statement is a special case of~\cite[Lemma 2.3]{MR4726503}. 
As for the second part of the statement, let $F \in O_K[T]$ be such 
that $f^k(T) = {(T^{2^{k - 1}} + f^{k-1}(0))}^2 + c + 4 F(T)$. Note 
that $\deg F \le 2^k$. Moreover, note that $\alpha \beta \in O_{L}$. 
Specialize the last equality at $T = \alpha$, multiply by $\beta^{2^k}$ 
to obtain
\[
    \beta^{2^k} f^k(\alpha) 
    = \beta^{2^k}{(\alpha^{2^{k - 1}} + f^{k-1}(0))}^2 
    + \beta^{2^k} c 
    + 4 \beta^{2^k} F(\alpha).
\]
Since $\deg F \le 2^k$, it follows that $\beta^{2^k} F(\alpha) \in 
O_L$ and the proof follows.
\end{proof}

The remaining lemmas are needed to study the integrality of the 
constant term of $f$ and its $2$-adic valuation. Consider the 
polynomial $F(T, C) = T^2 + C \in \ZZ[T, C]$. Define the iterates 
$F^1(T, C) = F(T, C)$ and $F^k(T, C) = F(F^{k - 1}(T, C), C)$ for $k 
\ge 2$. Then $F^k(T, c) = f^k(T)$ for all $k \ge 1$.

\begin{lemma}\label{lemma:integrality-of-c}
Let $f(T) = T^2 + c \in \overline \QQ[T]$ be PCF. Then $c$ is an 
algebraic~integer. 
\end{lemma}
\begin{proof}
Since $f$ is PCF, there exist positive integers $n < m$ such that $f^n(0) = f^m(0)$. 
Thus $F^m(0, c) = F^n(0, c)$, so $c$ is a root of the polynomial
\(
F^m(0, C) - F^n(0, C) \in \ZZ[C].
\)
Since $m > n$, the polynomial $F^m(0, C) - F^n(0, C)$ is monic of 
degree $2^{m - 1}$. Hence, $c$ is an algebraic integer.
\end{proof}

\begin{lemma}\label{lemma:F-properties}
For all $k \ge 1$, the polynomial $F^k(0, C)$ is divisible by $C$ and 
$\deg_C F^k(0, C) = 2^{k - 1}$.
\end{lemma}
\begin{proof}
We proceed by induction on $k \ge 1$. If $k = 1$, we have $F^1(0, C) = C$, 
which is divisible by $C$ and has degree $\deg_C F^1(0, C) = 1$. If 
$k \ge 2$, assume that $F^{k - 1}(0, C)$ is divisible by $C$ and that 
$\deg_C F^{k - 1}(0, C) = 2^{k - 2}$. In particular, there exists 
$G_{k - 1} \in \ZZ[C]$ such that $F^{k - 1}(0, C) = C G_{k - 1}(C)$. 
Then
\begin{equation}\label{eqn:Fk}
    F^{k}(0, C) 
    = {F^{k - 1}(0, C)}^2 + C 
    = C^2 G_{k - 1}(C)^2 + C 
    = C(C G_{k - 1}(C)^2 + 1).
\end{equation}
So $F^{k}(0, C)$ is divisible by $C$. As for the degree, since
\[
    F^k(0, C) = F^{k - 1}(0, C)^2 + C 
\]
and by inductive hypothesis $\deg_C F^{k - 1}(0, C)^2 = 2^{k - 1} > 1$, 
it follows that 
\[
    \deg_C F^k(0, C) 
    = \deg_C \left( F^{k - 1}(0, C)^2 + C \right)
    = 2 \deg_C F^{k - 1}(0, C)
    = 2 \cdot 2^{k - 2}
    = 2^{k - 1}.
\]
The proof follows.
\end{proof}

For all $k \in \NN$, define the polynomial 
\[
    A_k(C) = F^{k + 1}(0, C) - F^k(0, C) \in \ZZ[C].
\]

\begin{lemma}\label{lemma:lemma-for-2-adic-valuation}
For each $k \ge 1$ there exists an irreducible $B_k \in \ZZ[C]$ such 
that $A_{k + 1}(C) = C B_k(C) A_k(C)$.
\end{lemma}
\begin{proof} 
We first establish the existence of such polynomials $B_k$. 
\Cref{lemma:F-properties} implies that for each $k \ge 1$ there 
exists $G_k \in \ZZ[C]$ such that $F^k(0, C) = C G_k(C)$. Therefore
\begin{align*}
A_{k + 1}(C)  &= F^{k + 2}(0, C) - F^{k + 1}(0, C)\\
        &= F^{k + 1}(0, C)^2 - F^k(0, C)^2\\
        &= (F^{k + 1}(0, C) + F^k(0, C)) A_k(C)\\
        &= C (G_{k + 1}(C) + G_k(C)) A_k(C).
\end{align*}
Define $B_k(C) = G_{k + 1}(C) + G_k(C)$.

Next we prove the irreducibility of $B_k$. We will prove that it 
satisfies Eisenstein's criterion at the prime $2$. It suffices to 
show that every nonleading coefficient of $B_k$ is even, and the 
constant term is not divisible by $4$.

Note that \Cref{lemma:F-properties} implies $\deg_C G_k = 2^{k - 1} - 1$. 
Since $\deg_C G_{k + 1} > \deg_C G_k$, we conclude $\deg_C B_k = 2^k - 1$.
To prove the congruence $B_k \equiv C^{2^k - 1} \mod 2\ZZ[C]$, we 
first note that an induction on $k \ge 1$ implies that
\[
	G_k(C) \equiv C^{2^{k - 1} - 1} + \cdots + C + 1 \mod 2 \ZZ[C].
\]
Therefore, we conclude that
\begin{align*}
	B_k
	&= G_{k + 1}(C) + G_k(C)\\
	&\equiv (C^{2^k - 1} + C^{2^{k - 1} - 1} + \cdots + C + 1) 
    + (C^{2^{k - 1} - 1} + \cdots + C + 1)\\
	&\equiv C^{2^k - 1} \mod 2\ZZ[C]
\end{align*}
for all $k \ge 1$.
As for the second condition to apply Eisenstein's criterion, it 
follows from the definition that $G_k(0) = 1$ for all $k \ge 1$. 
Therefore $B_k(0) = 2 \neq 0 \mod 4 \ZZ[C]$. 

Hence, every nonleading coefficient of $B_k$ is even, and the 
constant term is not divisible by $4$. By Eisenstein's criterion, the 
polynomial $B_k$ is irreducible.\qedhere
\end{proof}

\subsection{Unramified case}

In this subsection, we write $L = K(\zeta)$ to simplify the notation. 
Suppose that $2$ is unramified in $L$. Recall that $d = [K \colon \QQ]$.

Let $\p$ be a prime ideal of $K$ above $2$. Let $\q$ be a prime ideal 
of~$L$ lying above~$\p$ and denote by $\artin{L}{K}{\q}$ the Artin 
symbol at~$\q$. This is defined as the Frobenius element at $\q$ in 
$\Gal(L / K)$, which is the unique Galois automorphism such that
\begin{equation}\label{eqn:def-frobenius}
    \artin{L}{K}{\q}(\alpha) \equiv \alpha^{2^{r(\p | 2)}} \mod \q
\end{equation}
for all $\alpha \in O_{L}$. The Frobenius elements at two distinct 
prime ideals lying above $\p$ are conjugate. Since $L / K$ is 
abelian, we deduce that $\artin{L}{K}{\q}$ does not depend on~$\q$, 
and we denote it by $\artin{L}{K}{\p}$. Since~\eqref{eqn:def-frobenius} 
holds for all prime ideals~$\q$ lying above~$\p$, we conclude that
\[
    \artin{L}{K}{\p} (\alpha)
    \equiv \alpha^{2^{r(\p | 2)}} \mod \p O_{L}
\]
for all $\alpha \in O_{L}$. Note that this uses the fact that $\q$ is 
unramified at $\p$.

We now obtain a congruence modulo $2 O_L$. First, note that $L / \QQ$ 
is a finite Galois extension and $\Gal(L / \QQ)$ is isomorphic to a 
subgroup of $\Gal(K/\QQ) \times \Gal(\QQ(\zeta) / \QQ)$, where the 
isomorphism is given by the restrictions. Under this embedding, the 
first component of $\sigma^d$ is trivial, since every element of 
$\Gal(K / \QQ)$ has order dividing $d$. Thus $\sigma^d$ lies in a 
subgroup isomorphic to $\{1\} \times \Gal(\QQ(\zeta) / \QQ)$. It 
follows that $\sigma^{d}$ lies in the center of $\Gal(L / \QQ)$ for 
all $\sigma \in \Gal(L / \QQ)$. Note that this property need not hold 
for an arbitrary abelian extension of $K$. Consider $\varphi_{\q} = 
\artin{L}{\QQ}{\q}^{d}$. If $\q$ and $\q'$ are two prime ideals in 
$L$ lying above $2$, the automorphism $\varphi_{\q}$ and 
$\varphi_{\q'}$ 
are conjugate inside $\Gal(L / \QQ)$. Since they lie in the center of 
$\Gal(L / \QQ)$, they must be equal. 
We therefore drop the dependence on $\q$ from the notation. Thus we 
obtain the congruence 
\begin{equation}\label{eqn:frob-q}
    \varphi(\alpha) 
    \equiv \alpha^{2^{d}} \mod \q
\end{equation}
for all $\alpha \in O_L$ and $\q$ lying above $2$. Since $L$ is 
unramified above $2$, we have $2 O_L = \prod_{\q | 2} \q$. Hence, 
since~\eqref{eqn:frob-q} holds for all prime ideals $\q$ above $2$, 
we conclude that
\begin{equation}\label{eqn:frobenius}
    \varphi(\alpha) 
    \equiv \alpha^{2^{d}} \mod 2 O_L.
\end{equation}
As a final remark, the restriction of $\varphi$ to $K$ is the $d$th 
power of an element in $\Gal(K / \QQ)$, hence it is trivial. It 
follows that $\varphi \in \Gal(L / K)$.

\begin{lemma}\label{lemma:idiot-computation}
Let $R$ be a commutative ring extension of $\ZZ$. If $x - y \in 2 R$, 
then ${(x + z)}^2 - {(y + z)}^2 \in 4 R$ for all $z \in R$.
\end{lemma}
\begin{proof}
Let $w \in R$ be such that $x - y = 2 w$. Then
\[
	{(x + z)}^2
	= {(y + 2w + z)}^2
	= {(y + z)}^2 + 4{(w^2 + w(y + z))}
\]
and the proof follows.
\end{proof}

\begin{prop}\label{prop:unramified-congruence}
Let $f(T) = T^2 + c \in O_K[T]$ be PCF. Choose positive integers $n 
< m$ such that $f^{n - 1}(0) = f^{m - 1}(0)$. 
Then for each $\alpha \in L$ there exists $\beta \in O_L$ such that
\begin{itemize}
    \item[(i.)] $|\beta|_v = {\max\{1, |\alpha|_v\}}^{-1}$ for all 
    places $v$ above $2$;
    \item[(ii.)] $\beta^{2^k} f^k(\alpha) \in O_L$ for all $k \ge 0$;
    \item[(iii.)] $\varphi^{m - n} \left(\beta^{2^{dn}}f^{dn}(\alpha)\right) 
    \equiv \beta^{2^{dm}}f^{dm}(\alpha) \mod 4 O_{L}$.
\end{itemize}
\end{prop}
\begin{proof}
\Cref{lemma:ad-denominator} applied with $L = K(\zeta)$ and $p = 2$ 
implies that there exists $\beta \in O_{L}$ such that $\alpha \beta 
\in O_L$ and $|\beta|_v = {\max\{1, |\alpha|_v\}}^{-1}$ for all 
places $v$ above~$2$, so point (i.) follows. As for point (ii.), we 
proceed by induction on $k \ge 0$. The case $k = 0$ follows from the 
above discussion. If $k \ge 1$, then 
\[
    \beta^{2^k} f^k(\alpha)
    = {(\beta^{2^{k - 1}} f^{k - 1}(\alpha))}^2 + \beta^{2^{k}}c
\]
and the proof follows since $c \in O_K$ and $\beta^{2^{k - 1}} f^{k - 1}(\alpha)$ 
lies in $O_L$ by induction on $k \ge 0$.
It remains to prove (iii.). The second part 
of~\Cref{lemma:iterates-of-f} applied with $k = dn$ implies that     
\[
    \beta^{2^{dn}}f^{dn}(\alpha) 
    = \beta^{2^{dn}}{\left(\alpha^{2^{dn - 1}} 
    + f^{dn - 1}(0)\right)}^2 + \beta^{2^{dn}}c 
    + 4 \gamma
\]
for some $\gamma \in O_L$.
Applying $\varphi^{m - n}$ to both sides yields
\begin{multline}\label{eqn:unramified-lemma}
    \varphi^{m - n} \left(\beta^{2^{dn}}f^{dn}(\alpha)\right)
    \equiv \left(\varphi^{m - n} (\beta\alpha)^{2^{dn - 1}} 
    + \varphi^{m - n} (\beta^{2^{dn - 1}})f^{dn - 1}(0)\right)^2\\
	+ \varphi^{m - n} (\beta^{2^{dn}})c
    \mod 4O_L.
\end{multline}
Here we used that $c, f^{dn - 1}(0), f^{dm - 1}(0) \in K$, so they 
are fixed by $\varphi$. Note that iterating~\eqref{eqn:frobenius} on 
$\beta \alpha$ gives $\varphi^{m - n}(\beta \alpha) \equiv 
(\beta \alpha)^{2^{d(m - n)}} \mod 2O_L$ and raising to $2^{dn - 1}$ 
yields $\varphi^{m - n}(\beta \alpha)^{2^{dn - 1}} \equiv 
(\beta \alpha)^{2^{dm - 1}} \mod 2O_L$. Then, 
\Cref{lemma:idiot-computation} applies with $R = O_L$ and 
\[
	(x, y, z) 
	= \left(\varphi^{m - n}(\beta \alpha)^{2^{dn - 1}}, 
	  (\beta \alpha)^{2^{dm - 1}},
	  \varphi^{m - n} (\beta^{2^{dn - 1}})f^{dn - 1}(0)\right)
\]
giving the congruence
\begin{multline*}
	{\left(\varphi^{m - n} (\beta\alpha)^{2^{dn - 1}} 
    + \varphi^{m - n} (\beta^{2^{dn - 1}})f^{dn - 1}(0)\right)}^2\\
	\equiv \left(
	\beta^{2^{dm - 1}}\alpha^{2^{dm - 1}} 
    + \varphi^{m - n} (\beta^{2^{dn - 1}})f^{dn - 1}(0)
    \right)^2 
    \mod 4O_{L}.
\end{multline*}
Substituting this last congruence into~\eqref{eqn:unramified-lemma} yields
\begin{multline}\label{eqn:unramified-lemma-2}
    \varphi^{m - n} \left(\beta^{2^{dn}}f^{dn}(\alpha)\right)\\
	\equiv \left(
	\beta^{2^{dm - 1}}\alpha^{2^{dm - 1}} 
    + \varphi^{m - n} (\beta^{2^{dn - 1}})f^{dn - 1}(0)
    \right)^2
    + \varphi^{m - n} (\beta^{2^{dn}})c
    \mod 4O_{L}
\end{multline}

The hypothesis on $m$ and $n$ implies that $f^{n - 1}(0) = f^{n - 1 + 
(m - n)}$. Applying further iterates of $f$ gives $f^j(0) = f^{j + m 
- n}(0)$ for every $j \ge n - 1$. Since $dn - 1 \ge n - 1$ and $dm - 
1 = (dn - 1) + d(m - n)$ it follows that $f^{dn - 1}(0) = f^{dm - 1}(0)$.

If $dn = 1$, then $d = n = 1$ and $f^{dm - 1}(0) = f^{dn - 1}(0) = 0$. 
Since $\varphi$ fixes $c$, we have
\[
    \varphi^{m-1}\left(\beta^2 f(\alpha)\right)
    = \varphi^{m-1}(\beta\alpha)^2
    + \varphi^{m-1}(\beta)^2c
    \equiv \beta^{2^m}\left(\alpha^{2^m}+c\right)
    \mod 4O_L,
\]
where the congruence follows from iterating~\eqref{eqn:frobenius} and 
squaring.
Since $f^{m - 1}(0) =~0$, \Cref{lemma:iterates-of-f} applied with 
$k = m$ gives $\beta^{2^m}(\alpha^{2^m} + c) \equiv 
\beta^{2^m}f^m(\alpha) \mod 4 O_L$. Thus $\varphi^{m - 1} \left(
\beta^{2}f(\alpha)\right) \equiv \beta^{2^m} f^m(\alpha) \mod 4O_L$, 
and the proof in this case follows.

If $dn \ge 2$, then a repeated application of the Frobenius implies 
that $\varphi^{m - n} (\beta^{2^{dn - 1}}) \equiv \beta^{2^{dm - 1}} 
\mod 4 O_{L}$, and squaring yields $\varphi^{m - n} (\beta^{2^{dn}}) 
\equiv \beta^{2^{dm}} \mod 4 O_{L}$. Thus~\eqref{eqn:unramified-lemma-2} becomes
\begin{equation*}
    \varphi^{m - n} \left(\beta^{2^{dn}}f^{dn}(\alpha)\right)
	\equiv \left(
	\beta^{2^{dm - 1}}\alpha^{2^{dm - 1}} + \beta^{2^{dm - 1}}f^{dm - 1}(0)
    \right)^2 + \beta^{2^{dm}}c
    \mod 4O_{L}
\end{equation*}

\Cref{lemma:iterates-of-f}, applied with $k = dm$, gives
\begin{equation*}
	\varphi^{m - n} \left(\beta^{2^{dn}}f^{dn}(\alpha)\right)
	\equiv \beta^{2^{dm}} f^{dm}(\alpha) \mod 4 O_{L}.
\end{equation*}
The proof follows.
\end{proof}

\begin{remark}
If $\alpha \in O_{K(\zeta)}$, then in the above proposition one can 
take $\beta = 1$. In particular, it follows that
\begin{equation}\label{eqn:eqn-unr-alg-int}
    \varphi^{m - n} \left(f^{dn}(\alpha)\right) \equiv f^{dm}(\alpha) \mod 4 O_{L}
\end{equation}
for all $\alpha \in O_{K(\zeta)}$.
\end{remark}

\begin{lemma}\label{lemma:unramified-equality-case}
Let $f \in O_K[T]$. Let $\alpha \in {K(\zeta)}$. If there exists $n < m$ such that
\[
    \varphi^{m - n} \left(f^{dn}(\alpha)\right)
    = f^{dm}(\alpha),
\] 
then $\alpha$ is preperiodic.
\end{lemma}
\begin{proof}
Let $D$ denote the degree of $f$. Since $f$ is defined over $K$, the 
canonical height is invariant under $\Gal(\overline K / K)$, and in 
particular under $\Gal(L / K)$; cf.~\cite[Theorems~3.6 and~3.20]{MR2316407}. 
By using the hypothesis together with property (i.) of the canonical 
height, we get
\[
    0 
    = \hat h_f\left(\varphi^{m - n} \left(f^{dn}(\alpha)\right)\right)
    - \hat h_f(f^{dm}(\alpha))
    = (D^{dn} - D^{dm}) \hat h_f(\alpha).
\]
As $n < m$, it follows that $\hat h_f(\alpha) = 0$. This is 
equivalent to $\alpha$ being preperiodic by the Dynamical Kronecker 
Theorem.
\end{proof}

\subsection{Ramified case}

In the following proofs, we will often reduce to the setting of 
$p$-adic local fields, which are finite extensions of $\QQ_p$. If 
$F / E$ is an extension of local fields, then there exists a unique 
extension $v_F$ of the valuation $v_E$ to $F$~\cite[Theorem 4.8]{MR1697859}. 
The \textit{ramification index} of~$F / E$ is defined as
the index $e(F / E) = [v_F(F^*) \colon v_E(E^*)]$. It follows from 
the definition that the ramification index is multiplicative, that 
is, if $H$ is an intermediate field of $F / E$, then $e(F / E) = e
(F / H) e(H / E)$. 
We say that $F / E$ is \textit{unramified} if $e(F / E) = 1$, 
\textit{ramified} if $e(F / E) \ge 2$ and \textit{totally ramified} 
if $e(F / E) = [F \colon E]$.
In some cases, there are explicit formulas to compute the 
ramification index. For example, let $\zeta$ be a root of unity and 
write its order as $2^e m$ with $m$ odd. If $e \ge 1$, then 
$e(\QQ_2(\zeta) / \QQ_2) = 2^{e - 1}$, cf.~\cite[Proposition~II.7.13]{MR1697859}.

Given a number field $K$ and a prime ideal $\p$ of $K$, we write 
$v_\p$ for the $\p$-adic valuation normalized such that $v_\p(p) = 1$ 
for the integer prime $p$ below $\p$. The completion $K_\p$ of $K$ at 
$v_\p$ is a local number field. We denote the extension of $v_\p$ to 
$K_\p$ by $\hat v_\p$ and the valuation ring of $K_\p$ by $O_\p = \{z 
\in K_\p \colon \hat v_\p(z) \ge 0\}$.

\begin{prop}\label{prop:local-fields}
Let $L / K$ be an extension of number fields. Let $\p$ be a prime 
ideal of $K$ and let $\q$ be a prime ideal of $L$ lying above $\p$. 
\begin{itemize}
    \item[(i.)] $L_\q = L K_\p$
    \item[(ii.)] There exists $\alpha \in O_\q$ such that $O_\q = 
    O_\p[\alpha]$.     If the extension is totally ramified, then one 
    can take $\alpha$ to be a uniformizer for $L_\q$.
    \item[(iii.)] $e(L_\q / K_\p) = e(\q | \p)$.
\end{itemize}
\end{prop}
\begin{proof}
See~\cite[Ch. II.8]{MR1697859} for point (i.) and see~\cite[Lemma II.10.4]{MR1697859}
and~\cite[Proposition~2.53]{MR3838349} for point (ii.). As for point 
(iii.), it is a fact that $v_\p(K^*) = \hat v_\p(K^*_\p)$, 
see~\cite[Ch. II.4]{MR1697859}. In other words, $\hat v_\q$ extends 
$v_\q$ with index $1$. Thus
\[
    e(L_\q / K_\p)
    = [\hat v_\q(L^*_\q) \colon \hat v_\p(K^*_\p)]
    = [v_\q(L^*) \colon v_\p(K^*)]
    = e(\q | \p).\qedhere
\]
\end{proof}

\begin{prop}\label{prop:composition-of-unramified}
Let $F_1 / E$ and $F_2 / E$ be two extensions of local fields inside 
an algebraic closure $\overline E / E$. If $F_1 / E$ is unramified, 
then $F_1 F_2 / F_2$ is unramified.
\end{prop}
\begin{proof}
See~\cite[Proposition II.7.2]{MR1697859}.
\end{proof}

From now on in this subsection, we assume that $K$ is a Galois number 
field unramified above $2$, that $\zeta$ is a root of unity, and that 
$f(T) = T^2 + c \in O_K[T]$ is PCF.

\begin{lemma}\label{lemma:totally-ramified-1}
If $2$ ramifies in $K(\zeta)$, then the extension $K(\zeta) / K(\zeta^2)$ 
is quadratic. Moreover, it is totally ramified at all prime ideals of 
$K(\zeta^2)$ lying above $2$.
\end{lemma}
\begin{proof}
We first prove that $K(\zeta)/K(\zeta^2)$ is quadratic.
The prime $2$ ramifies in~$\QQ(\zeta)$. If it did not, then both $K / \QQ$ 
and $\QQ(\zeta)/\QQ$ would be unramified above $2$, and so the same 
would be true for $K(\zeta)$, a contradiction.
Therefore,~\cite[Corollary I.10.4]{MR1697859} implies that the order 
of $\zeta$ can be written as $2^e m$ with $e \ge 2$ and $m$ odd. It 
follows that there exist two roots of unity $\zeta_{2^e}$ and 
$\zeta_m$ of order $2^e$ and $m$ respectively such that $\QQ(\zeta) = 
\QQ(\zeta_{2^e})\QQ(\zeta_m)$. Since every intermediate field of 
$K(\zeta_m) / \QQ$ is unramified above $2$ while every intermediate 
field of $\QQ(\zeta_{2^e})/\QQ$ is totally ramified above $2$, we obtain 
$K(\zeta_m) \cap \QQ(\zeta_{2^e}) = K(\zeta_m) \cap \QQ(\zeta_{2^{e - 1}}) = \QQ$. 
Therefore $[K(\zeta) \colon \QQ(\zeta_{2^e})] = [K(\zeta^2) \colon \QQ
(\zeta_{2^{e - 1}})]$ and 
\[
    [K(\zeta) \colon \QQ(\zeta_{2^{e - 1}})]
    = [\QQ(\zeta_{2^e}) \colon \QQ(\zeta_{2^{e - 1}})] 
    \cdot
    [K(\zeta) \colon \QQ(\zeta_{2^e})]
    = 2 [K(\zeta^2) \colon \QQ(\zeta_{2^{e - 1}})].
\]
Hence
\begin{equation*}
    [K(\zeta) \colon K(\zeta^2)]
    = \frac{[K(\zeta) \colon \QQ(\zeta_{2^{e - 1}})]}{[K(\zeta^2) 
    \colon \QQ(\zeta_{2^{e - 1}})]}
    = 2.
\end{equation*}

Now we prove that $K(\zeta) / K(\zeta^2)$ is totally ramified at all 
prime ideals above $2$. 
By~\Cref{prop:local-fields}, it suffices to show that 
$K_\p(\zeta)/K_\p(\zeta^2)$ is totally ramified for all prime ideals 
$\p$ of $K$ above $2$.
The extension $\QQ_2(\zeta)/\QQ_2(\zeta^2)$ is quadratic and totally 
ramified. 
Applying \Cref{prop:composition-of-unramified} with $E = \QQ_2$, $F_1 = K_\p$ 
and $F_2 = \QQ_2(\zeta^2)$ implies that $K_\p(\zeta^2)/\QQ_2(\zeta^2) 
= K_\p \QQ_2(\zeta^2)/\QQ_2(\zeta^2)$ is also unramified. Hence, by 
multiplicativity of ramification indices,
\begin{align*}
    e(K_\p(\zeta) / K_\p(\zeta^2))
    &= \frac{e(K_\p(\zeta) / \QQ_2(\zeta^2))}{e(K_\p(\zeta^2) / \QQ_2(\zeta^2))}\\
    &= \frac{e(K_\p(\zeta) / \QQ_2(\zeta)) e(\QQ_2(\zeta) / \QQ_2(\zeta^2))}{e(K_\p(\zeta^2) / \QQ_2(\zeta^2))}\\
    &= \frac{1 \cdot 2}{1}
    = 2.
\end{align*}
Moreover, $K_\p(\zeta)/K_\p(\zeta^2)$ has degree at most $2$, and it 
is nontrivial because $\p$ ramifies in $K(\zeta)$. Therefore its 
degree is $2$, so its ramification index equals its degree. Thus 
$K_\p(\zeta) / K_\p(\zeta^2)$ is totally ramified.\qedhere
\end{proof}

\begin{prop}\label{prop:ramified-easy-case}
Suppose that $2$ ramifies in $K(\zeta)$. Let $\kappa$ be the 
nontrivial element in $\Gal(K(\zeta) / K(\zeta^2))$. Then 
$\kappa(\alpha) \equiv \alpha \mod 2O_{K(\zeta)}$ for all $\alpha \in 
O_{K(\zeta)}$. Moreover, $\kappa(f(\alpha)) \equiv f(\alpha) \mod 4O_
{K(\zeta)}$ for all $\alpha \in O_{K(\zeta)}$.
\end{prop}

\begin{proof}
The extension $K(\zeta) / K(\zeta^2)$ is quadratic by~\Cref{lemma:totally-ramified-1}. 
Let $\kappa$ be the generator of $\Gal(K(\zeta) / K(\zeta^2))$. We 
want to show that $\kappa(\alpha) \equiv \alpha \mod 2O_{K(\zeta)}$ 
for all $\alpha \in O_{K(\zeta)}$. First note that $\kappa(\zeta) = - \zeta$. 
Indeed, since ${\kappa(\zeta)}^2 = \kappa(\zeta^2) = \zeta^2$, we 
have $\kappa(\zeta) = \pm \zeta$. As $\zeta \notin K(\zeta^2)$, we 
conclude that~$\kappa(\zeta) = - \zeta$.
Let $\p$ be a prime ideal in $K$ above $2$. Let $\q$ be any prime 
ideal of $K(\zeta)$ lying above $\p$ and denote $\q \cap K(\zeta^2)$ 
by $\q'$. We work over the local cyclotomic extension 
$K_{\p}(\zeta) / K_{\p}(\zeta^2)$. We claim that $O_\q = O_{\q'}[\zeta]$. 
Indeed, write the order of $\zeta$ as $2^e m$ with $e \ge 2$ and $m$ 
odd, and decompose $\zeta = \zeta_{2^e} \zeta_m$ where $\zeta_{2^e}$ 
has order $2^e$ and $\zeta_m$ has order $m$. Since $\zeta_m$ has odd 
order, $\zeta_m \in K_\p(\zeta^2)$, and clearly $\zeta_{2^e}^2 \in K_\p(\zeta^2)$. 
Hence, we have $K_\p(\zeta^2, \zeta_{2^e}) = K_\p(\zeta^2, \zeta)$. 
Now $1 - \zeta_{2^e}$ is a uniformizer of $K_\p(\zeta)$. 
Thus,~\Cref{prop:local-fields} yields $O_\q = O_{\q'}[1 - \zeta_{2^e}] 
= O_{\q'}[\zeta_{2^e}]$. Since $\zeta_m$ is invertible in $O_{\q'}$ 
and $\zeta = \zeta_{2^e} \zeta_m$, we conclude that $O_\q = 
O_{\q'}[\zeta_{2^e}] = O_{\q'}[\zeta]$, proving the claim.
Since $\alpha \in O_{K(\zeta)} \subseteq O_\q$, there exists 
$\beta_1, \beta_2 \in O_{\q'}$ such that $\alpha = \beta_1 + \beta_2 
\zeta$. Therefore,
\begin{equation}\label{eqn:2-adic-value}
	\kappa(\alpha) - \alpha = -2 \beta_2 \zeta \in 2O_\q.
\end{equation}
Since~\eqref{eqn:2-adic-value} holds for all ideals $\q$ lying above 
$2$, we conclude that 
\[
    \kappa(\alpha) - \alpha 
    \in O_{K(\zeta)} \cap \bigcap_{\q | 2} 2O_{\q} 
    = 2O_{K(\zeta)}.
\]

As for the second part of the statement, write $\kappa(\alpha) = 
\alpha + 2 \beta$ for some $\beta \in O_{K(\zeta)}$. Hence
\[
	\kappa(f(\alpha)) 
	= {\kappa(\alpha)}^2 + c
	= \alpha^2 + c + 4(\beta^2 + \alpha\beta)
	= f(\alpha) + 4(\beta^2 + \alpha\beta)
\]
and the proof follows.
\end{proof}

\begin{prop}\label{prop:ramified-easy-general-case}
Suppose that $2$ ramifies in $K(\zeta)$. Let $\kappa \in 
\Gal(K(\zeta) / K)$ be as in~\Cref{prop:ramified-easy-case}. 
For each $\alpha \in K(\zeta)$ there exists $\beta \in O_{K(\zeta)}$ 
such that
\begin{itemize}
\item[(i.)] $|\beta|_v = {\max\{1, |\alpha|_v\}}^{-1}$ for all places $v$ above $2$;
\item[(ii.)] $\alpha \beta \in O_{K(\zeta)}$ and $\beta^2 f(\alpha) \in O_{K(\zeta)}$;
\item[(iii.)] $\kappa(\beta^2 f(\alpha)) \equiv \beta^2 f(\alpha) \mod 4O_{K(\zeta)}$.
\end{itemize}
\end{prop}

\begin{proof}
\Cref{lemma:ad-denominator} applied with $L = K(\zeta)$ and $p = 2$ 
implies that for each $\alpha \in K(\zeta)$ there exists $\beta \in O_{K(\zeta)}$ 
such that $\alpha \beta \in O_{K(\zeta)}$ and $|\beta|_v = \max\{1, 
|\alpha|_v\}^{-1}$ for all places $v$ above $2$. Point (i.) follows.
As for (ii.), we first show that $\beta^2 f(\alpha) \in O_{K(\zeta)}$. 
Indeed
\[
    \beta^2 f(\alpha) = \beta^2 \alpha^2 + \beta^2 c
\]
and the proof follows since $\beta, \beta\alpha \in O_{K(\zeta)}$ and 
$c \in O_K$.
To prove (iii.), we first claim that for all $\gamma \in O_{K(\zeta)}$ 
we have $\kappa(\gamma^2) \equiv \gamma^2 \mod 4O_{K(\zeta)}$. 
Indeed, by~\Cref{prop:ramified-easy-case} there exists $\tilde{\gamma} 
\in O_{K(\zeta)}$ such that $\kappa(\gamma) = \gamma + 2\tilde{\gamma}$. 
So 
\[
    \kappa(\gamma^2) = \gamma^2 + 4\tilde{\gamma}(\gamma + \tilde{\gamma})
\]
and the claim follows. We then conclude that 
\[
    \kappa(\beta^2f(\alpha))
    = \kappa(\beta^2\alpha^2) + \kappa(\beta^2)c
    \equiv \beta^2\alpha^2 + \beta^2c 
    \equiv \beta^2 f(\alpha) \mod 4O_{K(\zeta)}.\qedhere
\]
\end{proof}

\begin{remark}
If $\alpha$ is an algebraic integer, one can take $\beta = 1$. It 
follows that
\begin{equation}\label{eqn:eqn-ram-alg-int}
    \kappa(f(\alpha)) \equiv f(\alpha) \mod 4O_{K(\zeta)}.
\end{equation}
\end{remark}

In the current setting, the analogue of \Cref{lemma:unramified-equality-case} 
is false. Indeed, for all elements $\alpha$ of the form $\zeta \beta$ with 
$\beta \in K(\zeta^2)$ one has $\kappa(f(\alpha)) = f(\alpha)$ and 
$f(\alpha) \in K(\zeta^2)$. This is the main difference with the 
unramified case. However, if $2$ ramifies in $K(\zeta^4)$, we are 
going to prove two facts. The first is that $f(\alpha) \notin 
K(\zeta^4)$. The second is that if $\alpha$ is $\q$-adically integral 
for some prime ideal $\q$ of $K(\zeta)$ above $2$, then $f^2(\alpha) \notin 
K(\zeta^4)$ whenever $f(\alpha) \in K(\zeta^2)$ and $f \notin \{T^2, 
T^2 - 2\}$.

\begin{lemma}\label{lemma:quartic-cyclic}
If $2$ ramifies in $K(\zeta^4)$, then the extension 
$K(\zeta)/K(\zeta^4)$ is quartic and cyclic.
\end{lemma}
\begin{proof}
We first show that the order of $\zeta$ is a multiple of $16$. The 
prime $2$ is ramified in $\QQ(\zeta^4)$ since it is unramified in $K$ 
and ramified in $K(\zeta^4)$ by hypothesis. This implies that the 
order of $\zeta^4$ is divisible by $4$. Hence, the order of $\zeta$ 
is divisible by $16$.
Denote the order of $\zeta$ by $n$ and write $n = 2^a b$ for $a \ge 4$ 
and $b$ odd. Note that $n/4$ is the order of $\zeta^4$. Let $\phi$ be 
the Euler totient. Then,
\[
    [\QQ(\zeta) \colon \QQ(\zeta^4)]
    = \frac{[\QQ(\zeta) \colon \QQ]}{[\QQ(\zeta^4) \colon \QQ]}
    = \frac{\phi(n)}{\phi(n / 4)}
    = \frac{2^{a - 1}\phi(b)}{2^{a - 3}\phi(b)}
    = 4.
\]
Since $i \in \QQ(\zeta^4)$, the automorphism defined by 
$\sigma(\zeta) = i \zeta$ lies in $\Gal(\QQ(\zeta)/\QQ(\zeta^4))$. 
Since $\sigma$ has order $4$ and the extension has degree $4$, it 
generates $\Gal(\QQ(\zeta)/\QQ(\zeta^4))$. Hence $\QQ(\zeta)/\QQ(\zeta^4)$ 
is quartic and cyclic.

Next, we show that $\QQ(\zeta) \cap K(\zeta^4) = \QQ(\zeta^4)$. On 
the one hand, the extension $\QQ(\zeta)/\QQ(\zeta^4)$ is totally 
ramified at all prime ideals above $2$. Thus, the extension 
$\QQ(\zeta) \cap K(\zeta^4)/\QQ(\zeta^4)$ is totally ramified at all 
primes above $2$. Consequently $[\QQ(\zeta) \cap K(\zeta^4) \colon 
\QQ(\zeta^4)] = e(\QQ(\zeta) \cap K(\zeta^4)/\QQ(\zeta^4))$. On the 
other hand, since $K$ is unramified above $2$, an application 
of~\Cref{prop:composition-of-unramified} with $E = \QQ_2$, 
$F_1 = K_\p$ and $F_2 = \QQ_2(\zeta^4)$ implies that the extension 
$K(\zeta^4)/\QQ(\zeta^4)$ is unramified at all prime ideals above $2$.
Hence $e(\QQ(\zeta) \cap K(\zeta^4)/\QQ(\zeta^4)) = 1$. The second 
claim follows.

It follows from the second claim that $\QQ(\zeta)$ and $K(\zeta^4)$ 
are linearly disjoint over $\QQ(\zeta^4)$. Then, restriction to 
$\QQ(\zeta)$ induces an isomorphism between $\Gal(K(\zeta)/K(\zeta^4))$ 
and $\Gal(\QQ(\zeta)/\QQ(\zeta^4))$. Since the latter is cyclic of 
order 4, the proof of the lemma follows. 
\end{proof}

Assume that $2$ ramifies in $K(\zeta^4)$. Consider a generator 
$\rho \in \Gal(K(\zeta) / K(\zeta^4))$. Since $\rho$ fixes $\zeta^4$ 
but not $\zeta^2$, we deduce that $\rho(\zeta^2) = -\zeta^2$. 
Similarly, $\rho^2(\zeta) = - \zeta$. Note that $\rho^2 = \kappa$.

\begin{lemma}\label{lemma:first-iterate-degeneration}
Suppose that $2$ ramifies in $K(\zeta^4)$. If $\alpha \in K(\zeta) 
\setminus K(\zeta^2)$ and $f(\alpha) \in K(\zeta^2)$, then $f(\alpha) 
\notin K(\zeta^4)$. 
\end{lemma}
\begin{proof}
The group $\Gal(K(\zeta) / K(\zeta^4))$ is cyclic of order $4$ by 
\Cref{lemma:quartic-cyclic}. Suppose for contradiction that 
$\rho(f(\alpha)) = f(\alpha)$. Then $\rho(\alpha)^2 + c = \alpha^2 + c$, 
from which $\rho(\alpha) = \pm \alpha$. Hence $\kappa(\alpha) = 
\rho^2(\alpha) = \alpha$, a contradiction. 
\end{proof}

We now briefly investigate the $2$-adic valuation of $c$, the 
constant coefficient of the dynamical system $f$. 
Let $\P$ be a prime ideal of $\QQ(c)$ above $2$. Define $v_\P$ to be 
the $\P$-adic valuation normalized such that $v_\P(2) = 1$.

\begin{prop}\label{prop:2-adic-valuation-of-c} 
Let $f(T) = T^2 + c \in \overline \QQ [T]$ be PCF. If $c \neq 0$, 
then either $c \in O_{\QQ(c)}^\times$ or $v_\P(c) = 1/[\QQ(c) \colon \QQ]$.
\end{prop}

\begin{proof}
Since $f$ is PCF, there exist $h, k \in \NN$ such that $f^{k + h}(0) 
= f^{k}(0)$. It follows that the polynomial
\(
	Q(C) = F^{k + h}(0, C) - F^k(0, C)
\)
vanishes at $C = c$. Our strategy consists of factoring $Q$ into a 
product of irreducible 2-Eisenstein polynomials and a polynomial with 
constant term 1. Since the minimal polynomial of $c$ divides $Q$, it 
must divide one of these factors, yielding the two alternatives.
Note that
\begin{equation}\label{eqn:poly}
	Q(C)
	= \sum_{i = 0}^{h - 1} \left(
		F^{k + i + 1}(0, C) - F^{k + i}(0, C)
	\right)
	= \sum_{i = 0}^{h - 1} A_{k + i}(C).
\end{equation}
A repeated application of \Cref{lemma:lemma-for-2-adic-valuation} 
yields
\begin{equation}\label{eqn:formula-Ak}
	A_{k + i}(C) = C^i A_k(C) B_k(C) B_{k + 1}(C) \cdots B_{k + i - 1}
\end{equation}
for all $i \in \{1, \dots, h - 1\}$.
Substituting~\eqref{eqn:formula-Ak} into~\eqref{eqn:poly} and 
factoring out $A_k$, we obtain
\[
	Q(C)
	= A_k(C) \cdot
	\left(
		1 + 
		\sum_{i = 1}^{h - 1} C^i \prod_{j = 1}^i B_{k + i - j}
	\right).
\]
Write $R(C) = 1 + \sum_{i = 1}^{h - 1} C^i \prod_{j = 1}^i B_{k + i - j}$.

Applying \Cref{lemma:lemma-for-2-adic-valuation} repeatedly once 
more, we obtain 
\[
	Q(C)
	= C^{k + 1} B_1(C) B_2(C) \cdots B_{k - 1}(C) R(C).
\]

Let $P \in \ZZ[C]$ be the minimal polynomial of $c$. From the above 
discussion, $P$ is an irreducible factor of $Q$. Since $c \neq 0$, it 
follows that $P(C) \neq C$. The factorization above implies that 
either $P$ equals $B_i$ for some $i \in \{1, 2, \dots, k - 1\}$ or it 
divides $R(C)$.

In the first case, it follows from the proof of~\Cref{lemma:lemma-for-2-adic-valuation} 
that $B_i$ is $2$-Eisenstein. Its reduction modulo $2$ is $B_i(C) 
\equiv C^{\deg B_i} \mod 2$. Since $B_i$ is the minimal polynomial of 
$c$, it follows that $\deg B_i = [\QQ(c) \colon \QQ]$. Finally,~\cite[Proposition~I.8.3]{MR1697859} 
implies that $2$ is totally ramified in $\QQ(c)$, from which $v_\P(c) 
= 1/[\QQ(c) \colon \QQ]$. 

In the second case, note that $R(C)$ has constant coefficient $1$. 
Therefore, the constant coefficient of any divisor of $R(C)$ in $\ZZ[C]
$ must divide $1$. Thus the constant coefficient of the minimal 
polynomial of $c$ is $\pm 1$. Hence, $c$ is an invertible algebraic 
integer.
\end{proof}

\begin{remark}\label{remark:c-valuation}
Since we assume that $2$ is unramified in $K$, if $c \notin\{0, -2\}$, 
then $v_\p(c) = 0$ for every prime ideal~$\p$ above $2$. Indeed, if 
$v_\p(c) = 1 / [\QQ(c) \colon \QQ]$ and $[\QQ(c) \colon \QQ] > 1$, 
then $2$ ramifies in $\QQ(c) \subseteq K$, a contradiction. 
Thus $[\QQ(c) \colon \QQ] = 1$. The only possible values of 
$c \in \QQ$ for which $T^2 + c$ is PCF are $\{0, -1, -2\}$ and the
only nonzero nonunit is $c = -2$.
\end{remark}

Let $\p$ be a prime ideal of $O_K$. Define $\overline{O_\p} = 
\{ \alpha \in \overline {K_\p} \colon \bar{v}_\p(\alpha) \ge 0\}$, 
where $\bar{v}_\p$ denotes the extension of the $\p$-adic valuation 
to $\overline{K_\p}$. Note that, for $\alpha \in K^\mathrm{cyc}$, the 
condition that $\alpha \in \overline{O_\p}$ for some prime ideal $\p$ lying 
above $2$ is equivalent to the existence of an embedding $\sigma 
\colon \QQ(\alpha) \rightarrow \overline{\QQ}_2$ such that $\sigma
(\alpha)$ is $2$-adically integral.

\begin{prop}\label{prop:second-iterate-degeneration}
Suppose that $2$ ramifies in $K(\zeta^4)$. Let 
$f(T) = T^2 + c \in K[T]$ be PCF with $c \notin \{ - 2, 0\}$. Suppose 
that $\alpha \in \overline{O_\p}$ for a prime ideal $\p$ of $K$ lying 
above $2$. If $\alpha \in K(\zeta) \setminus K(\zeta^2)$ and 
$f(\alpha) \in K(\zeta^2)$, then $f^2(\alpha) \notin K(\zeta^4)$.
\end{prop}
\begin{proof}
Recall that $\rho$ denotes the generator of $\Gal(K(\zeta) / K(\zeta^4))$ 
and that $\kappa = \rho^2$ denotes the generator of $\Gal(K(\zeta)/K(\zeta^2))$. 
Suppose for contradiction that $f^2(\alpha) \in K(\zeta^4)$. This means 
equivalently that $\rho(f^2(\alpha)) = f^2(\alpha)$. Note that 
\begin{align*}
    f^2(T_1) - f^2(T_2)
    &= (T_1^2 + c)^2 - (T_2^2 + c)^2\\
    &= T_1^4 - T_2^4 + 2c (T_1^2 - T_2^2)\\
    &= \left(T_1^2 - T_2^2\right) \left(T_1^2 + T_2^2 + 2c\right).
\end{align*}
Therefore, setting $T_1 = \rho(\alpha)$ and $T_2 = \alpha$ we get
\begin{equation}\label{eqn:factors}
    0
    = \rho(f^2(\alpha)) - f^2(\alpha)
    = ({\rho(\alpha)}^2 - \alpha^2)({\rho(\alpha)}^2 + \alpha^2 + 2 c).
\end{equation}
The first factor cannot vanish. Indeed, if it did, it would follow 
that $\rho(\alpha) = \pm \alpha$, so $\kappa(\alpha) = \rho^2(\alpha) 
= \alpha$, contradicting $\alpha \notin K(\zeta^2)$.
We conclude that ${\rho(\alpha)}^2 + \alpha^2 + 2 c = 0$. 

We claim that there exists $\beta \in K(\zeta^2)$ such that $\alpha = 
\zeta \beta$. It follows from the hypothesis that $\kappa(f(\alpha)) 
= f(\alpha)$ and $\kappa(\alpha) \neq \alpha$. The same argument as 
before implies that $\kappa(\alpha) = - \alpha$. Since $K(\zeta)$ is 
a $K(\zeta^2)$-vector space of dimension $2$, there exist $\beta', 
\beta \in K(\zeta^2)$ such that $\alpha = \beta' + \zeta \beta$. 
Applying $\kappa$ we get
\[
    \beta' - \zeta \beta 
    = \kappa(\alpha) 
    = - \alpha 
    = -\beta' - \zeta \beta.
\]
and conclude that $\beta' = 0$. Hence, $\alpha = \zeta \beta$. 
In particular, $\beta = \alpha \zeta^{-1}$ is $\q$-adically integral 
for some prime ideal $\q$ of $O_{K(\zeta)}$ above $\p$. 

Substitute the expression $\alpha = \zeta \beta$ in the second factor 
of~\eqref{eqn:factors} to get  
\[
    0 = {\rho(\alpha)}^2 + \alpha^2 + 2 c
      = -\zeta^2 \rho(\beta^2) + \zeta^2 \beta^2 + 2c.
\]

Write $\p' = \q \cap K(\zeta^2)$ and $\p'' = \q \cap K(\zeta^4)$.
Consider the local extensions $K(\zeta^4)_{\p''} \subseteq K(\zeta^2)_{\p'} 
\subseteq K(\zeta)_\q$. Applying the argument of~\Cref{prop:ramified-easy-case}, 
with $\zeta$ replaced by $\zeta^2$, we have $O_{\p'} = O_{\p''}[\zeta^2]$.
We claim that $\rho(\gamma) \equiv \gamma \mod 2O_{\p'}$ for all 
$\gamma \in O_{\p'}$. Indeed if $\gamma = \gamma_1 + \zeta^2 
\gamma_2$ for $\gamma_1, \gamma_2 \in O_{\p''}$, then $\rho(\gamma) - 
\gamma = -2 \zeta^2 \gamma_2 \in 2 O_{\p'}$. 

Since $\beta \in O_{\p'}$, we can write $\rho(\beta) = \beta + 2 
\delta$ for some $\delta \in O_{\p'}$. Therefore
\[
    0 = -\zeta^2 \rho(\beta^2) + \zeta^2 \beta^2 + 2c
      = -\zeta^2 {(\beta + 2 \delta)}^2 + \zeta^2 \beta^2 + 2c
      = -4 \zeta^2 \beta \delta - 4\zeta^2 \delta^2 + 2c.
\]
We conclude that 
\begin{equation}\label{eqn:expression}
    \zeta^2 \beta \delta + \zeta^2 \delta^2 = \frac{c}{2}. 
\end{equation}
The left-hand side lies in $O_{\p'}$. In~\Cref{remark:c-valuation} we 
proved that if $c \notin \{-2, 0\}$, then $c$ is a $\p$-adic unit. 
Thus $v_{\p}(c/2) = -1$, as $v_\p(2) = 1$. This contradicts the 
integrality of the left-hand side of~\eqref{eqn:expression}. We 
conclude that $f^2(\alpha) \notin~K(\zeta^4)$.\qedhere
\end{proof}

We can characterize all $\alpha \in K(\zeta)$ such that $f(\alpha) 
\in K(\zeta^2)$ and $f^2(\alpha) \in K(\zeta^4)$ as follows.

\begin{lemma}\label{lemma:characterize} 
Suppose that $2$ ramifies in $K(\zeta^4)$.
Let $f(T) = T^2 + c \in K[T]$ be PCF and suppose that $c \neq 0$. If 
$\alpha \in K(\zeta) \setminus K(\zeta^2)$, $f(\alpha) \in K(\zeta^2)
$ and $f^2(\alpha) \in K(\zeta^4)$, then there exists $\tilde \alpha 
\in K(\zeta^4) \setminus \{0\}$ such that $\alpha = \zeta 
\tilde{\alpha} - c / (2\zeta\tilde{\alpha})$. 
\end{lemma}

\begin{proof}
By \Cref{lemma:first-iterate-degeneration}, we have $f(\alpha) \in 
K(\zeta^2) \setminus K(\zeta^4)$. Apply the same argument as 
in~\Cref{prop:second-iterate-degeneration} twice: once to $\alpha$ 
(as $f(\alpha) \in K(\zeta^2)$) and once to $f(\alpha)$ (as 
$f^2(\alpha) \in K(\zeta^4))$ to get $\beta \in K(\zeta^2)$ and 
$\gamma \in K(\zeta^4)$ such that $\alpha = \zeta \beta$ and 
$f(\alpha) = \zeta^2 \gamma$. 

The hypothesis on the ramification of $2$ yields $K(\zeta^2) \neq 
K(\zeta^4)$. On the other hand $\zeta^2$ is a zero of $T^2 - \zeta^4 
\in K(\zeta^4)[T]$, so $[K(\zeta^2) \colon K(\zeta^4)] = 2$. Thus 
$\{1, \zeta^2\}$ is a basis of $K(\zeta^2)$ over $K(\zeta^4)$, and we 
may write $\beta = \beta_1 + \zeta^2 \beta_2$ for $\beta_1, \beta_2 
\in K(\zeta^4)$.

Using $\alpha = \zeta \beta$ and $f(\alpha) = \zeta^2 \gamma$, we 
compute 
\[
    \zeta^2\gamma
    = f(\alpha)
    = (2\zeta^4 \beta_1 \beta_2 + c) + \zeta^2(\beta_1^2 + \zeta^4 
    \beta_2^2).
\]
Comparing the coefficients with respect to the basis $\{1, \zeta^2\}$ 
of $K(\zeta^2)$ gives $2 \zeta^4 \beta_1 \beta_2 + c = 0$ and 
$\beta_1^2 + \zeta^4 \beta_2^2 = \gamma$. If $\beta_1 = 0$, the first 
equation gives $c = 0$, a contradiction. Hence, $\beta_1 \neq 0$ and 
$\beta_2 = - c/(2\zeta^4\beta_1)$. Thus 
\[
    \alpha 
    = \zeta\left(\beta_1 - \zeta^2 \frac{c}{2 \zeta^4 \beta_1}\right)
    = \zeta \beta_1 - \frac{c}{2 \zeta \beta_1}.
\]
Taking $\tilde{\alpha} = \beta_1$ proves the claim.
\end{proof}

We summarize the unramified and ramified cases as follows. Recall 
that in the unramified case $\varphi$ is the $d$th power of the 
Frobenius element at $2$, where $d = [K \colon \QQ]$, and the 
positive integers $m$ and $n$ are chosen such that $f^{n - 1}(0) = 
f^{m - 1}(0)$. In the ramified case we denote by $\kappa$ the 
generator of $\Gal(K(\zeta) / K(\zeta^2))$. We set the notation
\[
    \left(\sigma, (h,k)\right)
    =
    \begin{cases}
        \left(
            \varphi^{m - n}, 
            (dn, dm)
        \right) 
        & 2 \text{ is unramified in } K(\zeta),\\
        \left(
            \kappa, 
            (1, 1)
        \right) 
        & 2 \text{ ramifies in } K(\zeta).
    \end{cases}
\]

\begin{lemma}\label{lemma:2-adic-gain}
For all $\alpha \in K(\zeta)$ and all 
places $v \in M_{K(\zeta)}$ lying above $2$, we have
\[
    \left| \sigma(f^h(\alpha)) - f^k(\alpha) \right|_v
    \le \frac{1}{4} \max\left\{
        1, 
        \left| \sigma(f^h(\alpha)) \right|_v
    \right\}
    \cdot\max\{
        1, 
        \left| \alpha \right|_v
    \}^{2^k}.
\]
\end{lemma}
\begin{proof}
Suppose first that $K(\zeta) / \QQ$ is unramified above $2$. Let $\beta 
\in O_{K(\zeta)}$ be as in~\Cref{prop:unramified-congruence}. Then 
$\alpha \beta$ and $\beta$ are both in $O_{K(\zeta)}$ and $\beta^{2^k} 
f^k(\alpha) \in O_{K(\zeta)}$ for all integers $k \ge 0$.
Write
\begin{equation}\label{eqn:eqn-lhs}
\left| \sigma(f^h(\alpha)) - f^k(\alpha) \right|_v
= 	\left|\beta^{-2^k}\right|_v\left| \beta^{2^k}\sigma(f^h(\alpha)) 
- \beta^{2^k}f^k(\alpha) \right|_v.
\end{equation}
After adding and subtracting $\sigma\left(\beta^{2^h}f^h(\alpha)
\right)$, the ultrametric inequality yields that the left-hand side 
of~\eqref{eqn:eqn-lhs} is at most
\begin{equation}\label{eqn:eqn-unr-syn}
    \left|\beta^{-2^k}\right|_v \cdot \max\left\{
	\left|\sigma(f^h(\alpha))\right|_v  
	\left|\sigma(\beta^{2^h}) - \beta^{2^k}\right|_v,
	\left|\sigma\left(\beta^{2^h}f^h(\alpha)\right) - \beta^{2^k}
    f^k(\alpha)\right|_v
	\right\}.
\end{equation}
Since $\beta$ is an algebraic integer, we have 
$\sigma(\beta) \equiv \beta^{2^{d(m - n)}} \mod 2 O_{K(\zeta)}$. 
Since $h = dn \ge 1$, raising to $2^h$ gives 
\[
    \sigma(\beta^{2^h})
    \equiv \beta^{2^{d(m - n) + h}}
    \equiv \beta^{2^{dm}}
    = \beta^{2^k}
    \mod 4 O_{K(\zeta)}.
\]
This yields
\[
	\left|\sigma(\beta^{2^h}) - \beta^{2^k}\right|_v \le \frac{1}{4}.
\]
\Cref{prop:unramified-congruence} implies that
\[
    \left|\sigma\left(\beta^{2^h}f^h(\alpha)\right) - \beta^{2^k}
    f^k(\alpha)\right|_v \le \frac{1}{4}.
\]
By~\Cref{prop:unramified-congruence} it follows that $|\beta^{-2^k}|
_v = \max\{1, |\alpha|_v\}^{2^k}$. Substituting the previous three 
estimates into~\eqref{eqn:eqn-unr-syn}, we obtain
\[
    \frac{1}{4} {\max\{1, |\alpha|_v\}}^{2^k}
    \max\left\{
	\left|\sigma(f^h(\alpha))\right|_v,
	1
	\right\},
\]
which is the desired estimate in the unramified case.

Next suppose that $K(\zeta) / \QQ$ is ramified above $2$. The 
argument is identical. By \Cref{prop:ramified-easy-general-case}, 
there exists $\beta \in O_{K(\zeta)}$ such that $\alpha \beta \in O_{K(\zeta)}$ 
and $\beta^2 f(\alpha) \in O_{K(\zeta)}$. 
Using~\eqref{eqn:eqn-ram-alg-int} and \Cref{prop:ramified-easy-general-case}, 
we obtain 
\[
	\left|
    \sigma(\beta^{2}) - \beta^{2}
    \right|_v 
    \le \frac{1}{4}
\]
and
\[
    \left|
    \sigma\left(\beta^{2} f(\alpha)\right) - \beta^{2}f(\alpha)
    \right|_v 
    \le \frac{1}{4}.
\]
The same computation as above then yields the desired estimate.
\end{proof}

\section{Dimitrov's method}\label{section:dimitrov}

Let $K$ be a number field unramified above $2$ and let $\zeta$ be a 
root of unity. In this section, we define an auxiliary power series 
in $K(\zeta)[\![1/T]\!]$ and study its analytic continuation in 
$\CC$. First, we recall the notion of transfinite diameter of a 
bounded subset of $\CC$. Then we construct the power series. Finally, 
we recall the rationality criterion of P\'olya--Bertrandias.

\subsection{Transfinite diameter}

Given a nonempty bounded subset $C \subset \CC$ and an integer $n\ge 2$, 
define
\[
	\mathrm{d}_n(C)
		= \sup_{z_1, \dots, z_n\in C}
	{\left(
		\prod_{1\le i < j \le n}
		|z_j - z_i| 
	\right)}^{\frac{2}{n(n-1)}}.
\]
The sequence ${(\mathrm{d}_n(C))}_{n\ge 2}$ is nonnegative and 
monotonically nonincreasing~\cite[see the proof of Theorem~5.5.2]{MR1334766}. 
Its limit $\mathrm{d}(C)$ is the \textit{transfinite diameter of} 
$C$. We set the transfinite diameter of the empty set to be zero.
Given $z, w \in \CC$, the transfinite diameter of the line segment
\[
	[z,w] = \{(1 - t)z + t w \colon 0 \le t \le 1\}
\] 
is $|z - w|/4$. The transfinite diameter of a disk coincides with its 
radius. For further details, see~\cite[Ch. 5]{MR1334766}.

\subsection{The function}

Suppose that $K$ is a Galois number field unramified above~$2$ and 
write $d = [K \colon \QQ]$. Let $f(T) = T^2 + c \in K[T]$ be PCF. Let 
$\zeta$ be a root of unity. Recall that if~$2$ is unramified in $K(\zeta)$, 
we denoted by $\varphi$ the $d$th power of the Frobenius element at 
$2$ and by $n$ and $m$ two positive integers such that $f^{n - 1}(0) 
= f^{m - 1}(0)$. Moreover, if~$2$ is ramified in $K(\zeta)$, then 
$\kappa$ denotes the generator of $\Gal(K(\zeta) / K(\zeta^2))$. 
Recall the notation 
\begin{equation}\label{eqn:sigma-notation}
	\left(\sigma, (h,k)\right)
	=
	\begin{cases}
		\left(\varphi^{m - n}, (dn, dm)\right) & 2 \text{ is unramified in } K(\zeta),\\
		\left(\kappa, (1, 1)\right) & 2 \text{ ramifies in } K(\zeta).
	\end{cases}
\end{equation}
Define $\Phi \in K(\zeta)[\![1/T]\!]$ to be the unique power series 
with constant term $1$ such that 
\begin{equation}\label{eqn:power-series}
	\Phi^2 = \frac{T - f^k(\alpha)}{T - \sigma(f^h(\alpha))}.
\end{equation}

\begin{lemma}\label{lemma:archimedean-extension} 
Let $a,b\in\CC$. There exists a holomorphic function $g\colon 
\CC\setminus [a,b]\rightarrow \CC$ such that ${g(z)}^2 = (z - a)/(z - b)$ 
and $\lim_{z \rightarrow \infty} g(z) = 1$.
\end{lemma}
\begin{proof}
If $a = b$, then take $g = 1$. Therefore, assume that $a$ and $b$ are 
different. Let $h(z) = (z - a)/(z - b)$. For every smooth closed 
curve $\gamma \subseteq \CC \setminus [a, b]$, it holds that $[a, b] 
\subseteq \CC \setminus \gamma$. Since $[a, b]$ is connected, the 
points $a$ and $b$ lie in the same connected component of $\CC 
\setminus \gamma$. So the winding numbers of $\gamma$ with respect to 
$a$ and $b$ are equal. Hence, the integral 
\[
	\int_{\gamma} \frac{h'(z)}{h(z)}
	= \int_{\gamma}\left(
		\frac{1}{z - a} - \frac{1}{z - b}
	\right)
\]
equals zero. Therefore, the function $h'/h$ admits a holomorphic 
primitive $H$ on $\CC \setminus [a, b]$. Then
\[
	\frac{d}{d z} (h e^{- H} )
	= h' e^{- H} - h e^{- H} H'
	= 0.
\]
Hence $h = C e^H$ for some $C \in \CC \setminus \{0\}$. Choose $c_0 
\in \CC$ such that $c_0^2 = C$ and define $g(z) = c_0 e^{H(z)/2}$. 
Then
\begin{align*}
g^2(z) = C e^{H(z)} = h(z) = \frac{z - a}{z - b}.
\end{align*}
Since 
\[
	\lim_{z \rightarrow \infty} g(z)^2 
	= \lim_{z \rightarrow \infty} h(z)
	= 1, 
\]
the function $g$ is bounded in a neighborhood of infinity. So $g$ 
extends holomorphically to infinity, and its value there is either 
$1$ or $-1$. Replacing $g$ with $-g$ if necessary, we obtain $\lim_{z 
\rightarrow \infty} g(z) = 1$. The proof follows.
\end{proof}

For each $\tau \colon K(\zeta) \rightarrow \CC$, define $\tau(\Phi)$ 
by applying $\tau$ coefficientwise.
\Cref{lemma:archimedean-extension} applied with $a = \tau(f^k(\alpha))$ 
and $b = \tau\sigma(f^h(\alpha))$ implies that $\tau(\Phi)$ extends 
to $\CC \setminus [\tau(f^k(\alpha)), \tau\sigma(f^h(\alpha))]$.

Define 
\[
	S 
	= \left\{
		v \in M_{K(\zeta)}^0 
		\colon |\alpha|_v > 1 
		\text{ or } 
		|\sigma(\alpha)|_v > 1
	\right\} 
	\cup 	
	\left\{
		v \in M_{K(\zeta)}^0 
		\colon v \text{ lies above } 2
	\right\}
\]
where $M^0_{K(\zeta)}$ denotes the set of nonarchimedean places of 
$K(\zeta)$. The set $S$ is finite since an algebraic number is 
nonintegral at only finitely many nonarchimedean places. Then define 
\[
	O_{S} = 
	\left\{
		\beta \in K(\zeta) \colon 
		\left|\beta\right|_v \le 1 
		\text{ for all finite } v \not \in S
	\right\}.
\]
We prove that $\Phi \in O_S[\![1/T]\!]$ and that it converges outside 
sufficiently large $v$-adic disks. Before doing so, recall that the 
binomial series
\begin{equation}\label{eqn:binomial}
	\Psi = \sum_{j \ge 0} \binom{1/2}{j} T^j \in \QQ[\![T]\!]
\end{equation}
has the property $\Psi^2 = 1 + T$ and that $\Psi(4 T) \in \ZZ[\![T]\!]$. 
The first property dates back to Newton; the second was proved by 
Dimitrov~\cite{dimitrov}.

\begin{lemma}\label{lemma:coefficients}
With notation as above, $\Phi \in O_{S}[\![1/T]\!]$.
\end{lemma}
\begin{proof}
Write
\begin{align}\label{eqn:computation-2}
{\Phi(T)}^2 
&= \frac{T - f^k(\alpha)}{T - \sigma(f^h(\alpha))}
= 1 + \frac{\sigma(f^h(\alpha)) - f^k(\alpha)}{T - \sigma(f^h(\alpha))}\notag\\
&= 1 + \frac{\sigma(f^h(\alpha)) - f^k(\alpha)}{T} 
	\sum_{i \ge 0} \frac{\sigma(f^h(\alpha))^{i}}{T^i}.
\end{align}

Therefore 
\begin{align*}
\Phi(T) &= 
\Psi\left(\frac{\sigma(f^h(\alpha)) - f^k(\alpha)}{T} 
	\sum_{i \ge 0} \frac{\sigma(f^h(\alpha))^{i}}{T^i}\right)\\
&= \sum_{j \ge 0} 
\binom{1/2}{j}
{\left(
	\frac{\sigma(f^h(\alpha)) - f^k(\alpha)}{T} 
	\sum_{i \ge 0} \frac{\sigma(f^h(\alpha))^{i}}{T^i}
\right)}^j.
\end{align*}

Let $v \not \in S$ be a nonarchimedean place. Since $f$ is PCF, $c$ 
is an algebraic integer by~\Cref{lemma:integrality-of-c}, so $|c|_v 
\le 1$. Also, by definition of $S$, $|\alpha|_v \le 1$ and 
$|\sigma(\alpha)|_v \le 1$. It follows by induction that 
$|f^h(\alpha)|_v \le 1$ and $|\sigma(f^h(\alpha))|_v \le 1$ for all 
$h \ge 0$. Hence $|\sigma(f^h(\alpha)) - f^k(\alpha)|_v \le 1$.
Since $v \nmid 2$, the binomial coefficients $\binom{1/2}{j}$ are 
also integral with respect to $v$. Therefore every coefficient of 
$\Phi$ is integral with respect to $v$. As this holds for every 
nonarchimedean place $v \not \in S$, we conclude that $\Phi \in O_S[\![1/T]\!]$.
\end{proof}

In the proof of \Cref{lemma:coefficients} we defined
\[
	\Phi = \Psi\left(
		\frac{\sigma(f^h(\alpha)) - f^k(\alpha)}{T - \sigma(f^h(\alpha))}
		\right).
\]
The $v$-adic radius of convergence of the binomial series~\eqref{eqn:binomial} 
is $\rho_v = 1$ if $v$ does not lie above $2$ and $\rho_v = 1/4$ 
otherwise, cf.~\cite[Section~2]{dimitrov}.

Define 
\begin{equation}\label{eqn:nonarchimedean-convergence-radii}
R_v =
\begin{cases}
	\max\{
		1, 
		\left| \sigma(f^h(\alpha)) \right|_v
	\}
	\cdot\max\{
		1, 
		\left| \alpha \right|_v
	\}^{2^k}
	&\text{ if $v$ lies above $2$,}\\
		\max\left\{
		\left|\sigma(f^h(\alpha))\right|_v,
		\left|f^k(\alpha)\right|_v
	\right\}
	& \text{ otherwise}.
\end{cases}
\end{equation}

We show that $\Phi$ converges in the complement of the $v$-adic disk 
$B_v = \{z \in \CC_v \colon \left| z \right|_v \le R_v\}$.
 
\begin{lemma}\label{lemma:nonarchimedean-extension}
For all nonarchimedean places $v \in S$, the series $\Phi$ converges 
in $\CC_v \setminus B_v$.
\end{lemma}
\begin{proof}
The binomial series $\Psi$ converges $v$-adically for all $z \in 
\CC_v$ such that $|z| < \rho_v$.
To get the convergence of $\Phi$, we therefore need to guarantee that 
\[
	\left|
		\frac{\sigma(f^h(\alpha)) - f^k(\alpha)}{z - \sigma(f^h(\alpha))}
	\right|_v 
	< \rho_v
\]
for all $z \in \CC_v \setminus B_v$.
If $z \in \CC_v \setminus B_v$, then $|z|_v > R_v \ge |\sigma(f^h(\alpha))|_v$ 
and the ultrametric inequality yields $\left| z \right|_v = \left| z 
- \sigma(f^h(\alpha))\right|_v$. If $v$ does not lie above $2$, then 
\begin{equation*}
	\frac{\left| \sigma(f^h(\alpha)) - f^k(\alpha) \right|_v}{|z|_v} 
	< \frac{\left| \sigma(f^h(\alpha)) - f^k(\alpha) \right|_v}{
		\max\left\{
			\left| \sigma(f^h(\alpha)) \right|_v,
			\left| \sigma(f^h(\alpha)) - f^k(\alpha)\right|_v
		\right\}}
	\le 1.
\end{equation*}
If $v$ lies above $2$ we use~\Cref{lemma:2-adic-gain} and obtain
\begin{equation*}
	\frac{\left| f^k(\alpha) - {\sigma(f^h(\alpha))} \right|_v}{\left| z \right|_v}
	< \frac{\dfrac{1}{4} \max\{
			1, 
			\left| \sigma(f^h(\alpha)) \right|_v
		\}
		\cdot\max\{
			1, 
			\left| \alpha \right|_v
		\}^{2^k}}{
		\max\{
			1, 
			\left| \sigma(f^h(\alpha)) \right|_v
		\}
		\cdot\max\{
			1, 
			\left| \alpha \right|_v
		\}^{2^k}
		}
	= \frac{1}{4}.
\end{equation*}
The proof follows.
\end{proof}

\subsection{Rationality criterion} 

The transfinite diameter plays a role in the rationality criterion of 
P\'olya and Bertrandias, which we now recall. Let $L$ be a number 
field. Given a power series $\Theta = \sum_{i \ge 0} a_i/T^i \in 
L[\![1/T]\!]$ and an embedding $\tau \colon L \rightarrow\CC$, define 
$\tau(\Theta) = \sum_{i \ge 0}\tau(a_i) / T^i$. For $v \in M_L^0$, 
set $d_v = [L_v \colon \QQ_p]$ where $p$ is the rational prime below $v$. 

The ring $L [\![1/T]\!]$ is complete with respect to the valuation 
$v_\infty$ such that $v_\infty(1/T) = 1$. Polynomials in $L[T]$ have 
nonpositive valuation and lie in the field of Laurent power series 
$L(\!(1/T)\!)$. The following version of P\'olya--Bertrandias theorem 
is a sufficient condition that allows one to determine whether a 
power series $\Theta$ lies in $L(T) \cap L[\![1/T]\!]$, 
cf.~\cite[Th\'eor\`eme 5.4.6]{MR0447195}.

\begin{thm}[P\'olya--Bertrandias]\label{thm:polya-bertrandias}
Let $L$ be a number field and consider a power series $\Theta(T) = 
\sum_{i \ge 0} a_i / T^i \in L [\![1/T]\!]$. If there exists a finite 
subset of nonarchimedean places $S \subseteq M_L^0$ such that
\begin{itemize}
	\item[(i.)] if $v \not \in S$, then $\left| a_i \right|_v \le 1$ 
	for all $i \ge 0$;
	\item[(ii.)] for each $\tau \colon L \rightarrow \CC$ there 
	exists a bounded set $B_\tau \subseteq \CC$ such that $\CC 
	\setminus B_\tau$ is a domain and $\tau(\Theta)$ defines a 
	function over $\CC $ that extends holomorphically on $\CC 
	\setminus B_\tau$;
	\item[(iii.)] for each $v \in S$, the power series $\Theta$ 
	converges in the complement of a disk of radius~$R_v$ centered at 
	the origin.
\end{itemize}
Then if
\[
	\prod_{\tau \colon L \rightarrow \CC} 
	\mathrm{d}(B_\tau) 
	\cdot
	\prod_{v \in S} 
	R_v^{d_v} < 1,
\]
we conclude that $\Theta \in L(T) \cap L[\![1/T]\!]$, i.e.\ it is a 
rational function.
\end{thm}
\section{Equidistribution and height inequality}\label{section:equidistribution}

The Weil height of an algebraic number can be bounded from above in 
terms of the canonical Call--Silverman height by property (ii.). 
In this section we apply equidistribution to derive a sharper upper 
bound for the archimedean part of the Weil height in terms of the 
archimedean part of the canonical Call--Silverman height. 
The proof relies on two theorems. The first is an equidistribution 
theorem of Baker and Hsia~\cite{MR2164622} for points of small 
canonical height. The second, which will play a pivotal role in the 
proof of~\Cref{thm:main-thm}, is Pritsker's estimate for the Mahler 
measure of the iterates of a polynomial~\cite{MR4651642}.

\subsection{Equidistribution}

Following Baker--Hsia~\cite{MR2164622}, we say that a 
sequence~${\left(S_n\right)}_{n \ge 1}$ of finite nonempty subsets 
of~$\CC$ \textit{is equidistributed with respect to a probability 
measure}~$\mu$ on $\CC$ if for every continuous bounded function 
$u\colon \CC \rightarrow~\RR$ one has
\[
	\lim_{n\rightarrow\infty} 
	\frac{1}{\vert S_n\vert}
	\sum_{z\in S_n} u(z)
	 = \int u d\mu.
\]

Let $f \in \CC[T]$ be of degree at least $2$. The \textit{equilibrium 
measure} of the Julia set~$\mathrm{J}_f \subset \CC$ is defined as 
the Borel probability measure $\mu$ on $\mathrm{J}_f$ that minimizes 
the integral 
\[
    I(\mu) 
    = -\int_{\mathrm{J}_f}\int_{\mathrm{J}_f}
    \log\vert z - w \vert d\mu(z)d\mu(w).
\]
For its existence, see~\cite[Theorem 3.3.2]{MR1334766}. Brolin~\cite{MR0194595} 
proved that the set $S_{n,z}$ of solutions $w \in \CC$ for 
$f^n(w) = z$ counted with multiplicity is equidistributed with 
respect to $\mu_f$ for all but at most one $z \in \CC$. In other 
words, $\mu_f$ describes the asymptotic random distribution of 
preimages of nonexceptional complex numbers under iteration by $f$. 
For more details, see~\cite[Chapter~3.3]{MR1334766}. Freire, Lopes 
and Ma\~n\'e~\cite{MR0736568}, and independently Ljubich~\cite{MR0741393} 
extended Brolin's theorem to rational functions~$f\in\CC(T)$.

Let $L$ be a number field and fix an algebraic closure $\overline{L}$. 
Let $f \in L[T]$ be a polynomial of degree at least $2$. Baker and 
Hsia~\cite{MR2164622} proved that the Galois orbits of a sequence of 
algebraic numbers whose canonical height tends to zero is 
equidistributed with respect to the equilibrium measure $\mu_f$.

\begin{thm}[Baker--Hsia, {\cite[see Corollary~4.6]{MR2164622}}]\label{thm:baker-hsia}
Let $L$ be a number field, $f \in L[T]$ be a polynomial of 
degree at least $2$ and let $\mu_f$ be its equilibrium measure. 
Let ${(\alpha_n)}_{n \ge 1}$ be a sequence of pairwise distinct 
algebraic numbers such that $\lim_{n\rightarrow \infty}
\hat h_f(\alpha_n)~=~0$. 
Then, for every continuous bounded function $u \colon \CC 
\rightarrow \RR$ we have
\[
	\lim_{n\rightarrow \infty} 
	\frac{1}{[L(\alpha_n)\colon\QQ]}
	\sum_{\sigma\colon L(\alpha_n)\rightarrow \CC}
	u(\sigma(\alpha_n)) 
	= \int u d\mu_f.
\]
\end{thm}
In what follows, we will apply~\Cref{thm:baker-hsia} only on the 
archimedean places.

Favre and Rivera-Letelier~\cite{MR2221116}, Baker and Rumely~\cite{MR2244226}, 
and Chambert-Loir~\cite{MR2244803} extended these results to rational 
functions in both the archimedean and nonarchimedean cases.

\subsection{Mahler measure and iterates}

The \textit{Mahler measure} of a nonzero polynomial $f(T) = a (T - z_1) 
\cdots (T - z_d) \in \CC[T]$ is defined as $M(f) = \left|a\right| 
\prod_{i = 1}^d \max \left\{ 1, \vert z_i\vert \right\}$. Pritsker 
derived the following estimate for the exponential growth of the 
Mahler measure of the iterates of a polynomial. The proof is based on 
Brolin's theorem and the estimates in~\cite{MR2900164}.

\begin{thm}[Pritsker, {\cite[see Theorems~1.1 and~1.4]{MR4651642}}]\label{thm:pritsker}
	Let $f\in\CC[T]$ be a monic polynomial of degree at least $2$. Then,
	\begin{equation}\label{equation:pritsker-theorem}
		\lim_{n\rightarrow\infty} 
		\frac{\log M(f^{n})}{\deg (f^n)}
		= \int_{\mathrm{J}_f} \log^+ \vert z\vert d\mu_f.
	\end{equation}
	Moreover, if $\mathrm{J}_f$ is connected and if $f$ is either an 
    odd or an even function, then
	\[
		\int_{\mathrm{J}_f} \log^+ \vert z\vert d\mu_f 
		\le 2\cdot\int_{1}^2 \frac{\log t}{\pi\sqrt{4-t^2}}dt
		= 0.3230\dots.
	\]
\end{thm}

\begin{remark}
The filled Julia set of $f \in \CC[T]$ is connected if and only if
it contains all the critical points of $f$~\cite[see Theorem~17.3]{MR2193309}. 
If $f(T) = T^2 + c \in \CC[T]$ is PCF, it follows that $\mathrm{J}_f$ 
is connected.
\end{remark}


As a consequence of~\Cref{thm:baker-hsia,thm:pritsker}, we obtain an 
upper bound for the archimedean part of the Weil height with a better 
constant.

For a dynamical system $f$ defined over $\CC$, the \textit{canonical 
local height} of $z \in \CC$ is defined as
\[
	\hat \lambda_{f}(z) =
	\lim_{n \rightarrow \infty}
	\frac{\log^+|f^n(z)|}{\deg f^{n}}
\]	
where $\log^+|z| = \log\max\{1, |z|\}$.

Let $L$ be a number field and suppose that $f \in L[T]$. For $\alpha 
\in \overline L$, define the \textit{archimedean part of the Weil 
height} as
\[
	h^\infty(\alpha) = 
	\frac{1}{[L(\alpha) \colon \QQ]} 
    \sum_{\tau \colon L(\alpha) \rightarrow \CC}
	\log^+|\tau(\alpha)|,
\]
the \textit{archimedean part of the canonical height as}
\[
	\hat h_f^\infty(\alpha) = 
	\frac{1}{[L(\alpha) \colon \QQ]} 
    \sum_{\tau \colon L(\alpha) \rightarrow \CC}
	\hat \lambda_{\tau(f)}(\tau(\alpha)).
\]
Define also the \textit{nonarchimedean part of the Weil height} as 
\[
    h^0(\alpha) = h(\alpha) - h^\infty(\alpha)
\]
and the \textit{nonarchimedean part of the canonical height} as
\[
    \hat h_f^0(\alpha) = \hat h_f(\alpha) - \hat h_f^\infty(\alpha).
\]
If $f \in O_L[T]$ is monic, it turns out that
\[
    \hat h_f^0(\alpha)
    = \frac{1}{[L(\alpha) \colon \QQ]} 
    \sum_{v \in M_{L(\alpha)}^0} d_v \log^+|\alpha|_v 
    = h^0(\alpha),
\]
see~\cite[Theorem~5.61 and Remark~5.62]{MR2316407}.

In what follows, for every embedding $\tau \colon L \rightarrow \CC$, 
we write $I_\tau = I_{\tau(f)} = \int_{J_{\tau(f)}} \log^+|z| 
\mathrm{d}\mu_{\tau(f)}$ and $I = \max_{\tau} I_\tau$.

\begin{remark}\label{remark:I}
If $f(T) = T^2 + c \in L[T]$ is PCF, then every conjugate $\tau(f)$ 
is again PCF. So $\mathrm{J}_{\tau(f)}$ is connected and $\tau(f)$ 
is even. Hence, the second part of~\Cref{thm:pritsker} implies that 
\[
    I_\tau \le 2\cdot\int_{1}^2 \frac{\log t}{\pi\sqrt{4-t^2}}dt = 0.3230\dots
\]
and therefore the same holds true for $I$.
\end{remark}

\begin{lemma}\label{lemma:lemma}
Let $L$ be a number field and let $f \in O_L[T]$ be a nonconstant 
polynomial of degree at least $2$. Fix an embedding $\tau \colon L 
\rightarrow \CC$. For every $\eta > 0$ there exists $\delta = 
\delta(\tau, \eta) > 0$ such that for every $\alpha \in \overline L$ 
for which $\hat h_f(\alpha) < \delta$ and $[\QQ(\alpha) \colon \QQ] > 
\delta^{-1}$, we have
\[
    \frac{1}{[L(\alpha) \colon L]}
    \sum_{\substack{\sigma \colon L(\alpha) \rightarrow \CC\\\sigma|_L = \tau}}
    \left(
        \log^+|\sigma(\alpha)| - \hat \lambda_{\tau(f)}(\sigma(\alpha))
    \right)
    < I_\tau + \eta.
\]
\end{lemma}
\begin{proof}
Define $u_\tau(z) = \log^+|z| - \hat \lambda_{\tau(f)}(z)$. First 
observe that $u_\tau$ is bounded and continuous by~\cite[Theorem 5.60(c)]{MR2316407}.
Suppose that the statement is false. So there exists $\tilde\eta > 0$ 
such that for all $\delta > 0$ there exists $\beta \in \overline L$ 
for which $\hat h_f(\beta) < \delta$, $[\QQ(\beta) \colon \QQ] > 
\delta^{-1}$ but 
\[
    \frac{1}{[L(\beta) \colon L]}
    \sum_{\substack{\sigma \colon L(\beta) \rightarrow \CC\\\sigma|_L = \tau}}
    u_{\tau}(\sigma(\beta))
    \ge I_\tau + \tilde\eta.
\]
Then set ${(\delta_n)}_{n \ge 1} = (1/n)_{n \ge 1}$. For each 
$\delta_n$, choose $\beta_n$ as above such that $\hat h_f(\beta_n) < 1/n$ 
and $[\QQ(\beta_n) \colon \QQ] > n$. 
Thus, up to discarding repeated elements, we find a sequence of 
pairwise distinct elements whose height tends to zero. This sequence 
contradicts~\Cref{thm:baker-hsia}. Indeed, on the one hand, we have 
\begin{equation*}\label{eqn:eqn2}
    \int u_{\tau} \mathrm{d}\mu_{\tau(f)}
    = \int \left( \log^+|z| - \hat \lambda_{\tau(f)}(z) \right) \mathrm{d}\mu_{\tau(f)}
    = \int \log^+|z| \mathrm{d}\mu_{\tau(f)}
    = I_\tau
\end{equation*}
where the second equality follows from the fact that $\mu_{\tau(f)}$ 
is supported on the filled Julia set of $\tau(f)$ while $\hat \lambda_{\tau(f)}$ 
vanishes on it. On the other hand,
\[
    \lim_{n \rightarrow \infty} 
    \frac{1}{[L(\beta_n) \colon L]}
    \sum_{\substack{\sigma \colon L(\beta_n) \rightarrow \CC\\\sigma|_L = \tau}}
    u_{\tau}(\sigma(\beta_n))
    \ge I_\tau + \tilde\eta
    > \int u_\tau \mathrm{d}\mu_{\tau(f)},
\]
a contradiction to~\Cref{thm:baker-hsia}.
\end{proof}

\begin{thm}\label{thm:weil-to-call-silverman-inequality}
Let $L$ be a number field and $f\in O_L[T]$ be a nonconstant monic 
polynomial of degree at least $2$. For every $\eta > 0$ there 
exists $\delta > 0$ such that for every $\alpha \in \overline{L}$ for 
which $\hat h_f(\alpha) < \delta$ and $[\QQ(\alpha)\colon\QQ] > 
\delta^{-1}$, we have 
\[
	h(\alpha) < \hat h_f(\alpha) + I + \eta.
\] 
\end{thm}
\begin{proof}
Write the archimedean part of the Weil height of $\alpha$ as
\begin{align}\label{eqn:eqn}
    h^\infty(\alpha) 
    &= \frac{1}{[L(\alpha) \colon \QQ]} 
    \sum_{\tau \colon L \rightarrow \CC}
    \sum_{\substack{\sigma \colon L(\alpha) \rightarrow \CC\\\sigma|_L = \tau}}
    \log^+|\sigma(\alpha)|.
\end{align}
Fix $\eta > 0$. For each embedding $\tau \colon L \rightarrow \CC$, 
apply~\Cref{lemma:lemma}. This implies that there exists $\delta = 
\min_\tau \delta(\tau)$ such that if $\hat h_f(\alpha) < \delta$ and 
$[\QQ(\alpha) \colon \QQ] > \delta^{-1}$, then
\[
    \frac{1}{[L(\alpha) \colon L]}
    \sum_{\substack{\sigma \colon L(\alpha) \rightarrow \CC\\\sigma|_L = \tau}}
    \log^+|\sigma(\alpha)|
    < \frac{1}{[L(\alpha) \colon L]}
    \sum_{\substack{\sigma \colon L(\alpha) \rightarrow \CC\\\sigma|_L = \tau}}
    \hat\lambda_{\tau(f)}(\sigma(\alpha)) + I + \eta.
\]
From~\eqref{eqn:eqn} and the last inequality we conclude that
\[
    h^\infty(\alpha)
    < \frac{1}{[L \colon \QQ]} 
    \sum_{\tau \colon L \rightarrow \CC}
    \left(
    \frac{1}{[L(\alpha) \colon L]}
    \sum_{\substack{\sigma \colon L(\alpha) \rightarrow \CC\\\sigma|_L = \tau}}
    \hat\lambda_{\tau(f)}(\sigma(\alpha)) + I + \eta
    \right)
    = \hat h_f^\infty(\alpha) + I + \eta.
\]

Since $f \in O_L[T]$ is monic, the finite part of the Weil height 
equals the finite part of the Call--Silverman height. Hence, 
\[
    h(\alpha) = h^\infty(\alpha) + h^0(\alpha)
              < \hat h_f^\infty(\alpha) + I + \eta + \hat h_f^0(\alpha)
              = \hat h_f(\alpha) + I + \eta.
\]
The proof follows.\qedhere
\end{proof}
\section{Proofs}\label{section:proof}

In this section, we prove~\Cref{thm:main-thm,thm:main-thm-unr,thm:main-thm-ram}. 
All three results follow from~\Cref{thm:main-technical-thm}, which 
shows that the existence of suitable Galois automorphisms implies a 
lower bound for the canonical height of wandering points.

\begin{thm}\label{thm:main-technical-thm}

Let $K$ be a number field unramified above $2$. Let $\zeta$ be a root 
of unity. Let $f(T) = T^2 + c \in O_K[T]$ be PCF.
Let $(\sigma,(h,k))$ be as in~\eqref{eqn:sigma-notation}.
There exists $\varepsilon = \varepsilon(K, f) > 0$ such that 
$\hat{h}_f(\alpha) \ge \varepsilon$ for all wandering $\alpha \in 
K(\zeta)$ for which $\sigma (f^h(\alpha)) \ne f^k(\alpha)$.
\end{thm}

\begin{proof}
Define as in~\eqref{eqn:power-series} the power series 
\[
    \Phi^2 = \frac{T - f^k(\alpha)}{T - \sigma(f^h(\alpha))}
\]
where $k$, $h$ and $\sigma$ are as in~\eqref{eqn:sigma-notation}. We 
check that $\Phi$ satisfies the hypotheses (i.), (ii.) and (iii.) of 
\Cref{thm:polya-bertrandias}. \Cref{lemma:coefficients} implies that 
$\Phi \in O_S[\![1/T]\!]$ where 
\[
    S = \left\{
        v \in M_{K(\zeta)}^0 \colon 
        |\alpha|_v > 1 \text{ or } |\sigma(\alpha)|_v > 1
    \right\} 
    \cup \{
        v \in M_{K(\zeta)}^0 \colon 
        v \text{ lies above } 2
    \}.
\]
Thus, $\Phi$ satisfies point (i.). Let $\tau \colon K(\zeta) 
\rightarrow \CC$ be an embedding. Then the function
\[
    \tau(\Phi)^2 = \frac{T - \tau (f^k(\alpha))}{T - \tau\sigma (f^h(\alpha))}.
\]
extends holomorphically to the complement of $B_\tau = 
[\tau\sigma (f^h(\alpha)), \tau (f^k(\alpha))]$ by
\Cref{lemma:archimedean-extension}. Thus, $\Phi$ satisfies point 
(ii.). Finally, for all $v \in S$, \Cref{lemma:nonarchimedean-extension} 
implies that $\Phi$ extends analytically to the complement of the 
$v$-adic ball $B_v$ of radius $R_v$, with $R_v$ as 
in~\eqref{eqn:nonarchimedean-convergence-radii}. So (iii.) holds. 
Since we are assuming that $\sigma(f^h(\alpha)) \neq f^k(\alpha)$, 
the function $\Phi$ is not rational. Thus, the contrapositive 
of~\Cref{thm:polya-bertrandias} implies that
\begin{equation}\label{eqn:eqn1}
    0 \le \sum_{\tau \colon K(\zeta) \rightarrow \CC} 
    \log \mathrm{d}\left([\tau\sigma (f^h(\alpha)), \tau (f^k(\alpha))]\right)
    + \sum_{v \in S} d_v \log R_v.
\end{equation}

Recall that the transfinite diameter of a segment in $\CC$ is one 
fourth of its Euclidean length. Note that for all positive $r, s \in 
\RR$, we have 
\[
    r + s 
    \le 2 \max\{ r, s \} 
    \le 2\max\{1, r\} \cdot \max\{1,s\}.
\]    
Using the triangle inequality and the above remarks, we obtain
\begin{multline}\label{eqn:eqn3}
    \sum_{\tau \colon K(\zeta) \rightarrow \CC} 
    \log \mathrm{d}([\tau\sigma (f^h(\alpha)), \tau (f^k(\alpha))])\\
    \le \sum_{\tau \colon K(\zeta) \rightarrow \CC} 
    \left( 
        \log^+|\tau\sigma (f^h(\alpha))|
        + \log^+ |\tau (f^k(\alpha))|
        - \log 2 
    \right).
\end{multline}


As for the second sum in~\eqref{eqn:eqn1}, we distinguish two cases. 
If $v$ divides $2$, then
\[
      \log R_v
    = \log^+|\sigma(f^h(\alpha))|_v + 2^k \log^+|\alpha|_v.
\]
The ultrametric inequality and the integrality of $c$ imply that $2 
\log^+|z|_v = \log^+|f(z)|_v$ for all $z \in \CC_v$ and all 
nonarchimedean places $v \in S$. Indeed, if $|z|_v \le 1$, then $|z^2 
+ c|_v \le \max\{|z^2|_v, |c|_v\} \le 1$ while if $|z|_v > 1$, then 
$|z^2 + c|_v = |z|_v^2$. 
An induction yields $2^k \log^+|z|_v = \log^+|f^k(z)|_v$ for all~$k 
\in \NN$. Therefore we obtain
\begin{equation}\label{eqn:vdiv2}
    \log R_v 
    = \log^+|\sigma(f^h(\alpha))|_v + \log^+|f^k(\alpha)|_v.
\end{equation}

If $v$ does not divide $2$, then 
\begin{equation}\label{eqn:vndiv2}
    \log R_v
    \le \log^+|\sigma(f^h(\alpha))|_v + \log^+|f^k(\alpha)|_v,
\end{equation}
where the inequality follows since 
\[
    \max\left\{
		\left|\sigma(f^h(\alpha))\right|_v,
		\left|f^k(\alpha)\right|_v
	\right\}
    \le 
    \max\left\{ 
        1,
		\left|\sigma(f^h(\alpha))\right|_v
    \right\} \cdot 
    \max\left\{
        1, 
		\left|f^k(\alpha)\right|_v
	\right\}.
\]

Substitute~\eqref{eqn:vdiv2} and~\eqref{eqn:vndiv2} into the second 
sum of~\eqref{eqn:eqn1} to obtain
\begin{equation}\label{eqn:finite-part}
    \sum_{v \in S} d_v \log R_v
    \le \sum_{v \in S} \left(
        d_v \log^+|\sigma(f^h(\alpha))|_v + d_v \log^+|f^k(\alpha)|_v
    \right).
\end{equation}

Substitute~\eqref{eqn:eqn3} and~\eqref{eqn:finite-part} into~\eqref{eqn:eqn1} 
and enlarge the sum over $S$ to all nonarchimedean places to conclude 
that
\begin{align}\label{eqn:polya-ineq}
    0 
    & \le [K(\zeta) \colon \QQ] \left(
        h(f^h(\alpha)) + h(f^k(\alpha)) - \log 2
    \right).
\end{align}
Note that we have used the Galois-invariance of the Weil height to 
eliminate~$\sigma$.

Now fix $\eta = (\log2 - 2I)/4$ where $I$ was defined before~\Cref{remark:I}. 
Since $f$ is even and PCF, the second part of~\Cref{thm:pritsker} 
implies that $\log2 - 2I > 0$, and so $\eta > 0$. Let $\delta = 
\delta(\eta)$ be as in~\Cref{thm:weil-to-call-silverman-inequality} 
applied with $L = K$ to $\eta$. To proceed we may assume that 
$\hat h_f(\alpha) < \delta/2^k$ as we may assume $\varepsilon < 
\delta/2^k$.

First, suppose $[\QQ(\alpha)\colon\QQ] \le d 2^k\delta^{-1}$, where 
$d = [K \colon \QQ]$. Since the canonical height of $\alpha$ is 
assumed to be bounded, it follows from property (ii.) of the 
canonical height that the Weil height of $\alpha$ is bounded as well. 
By Northcott's theorem, $\alpha$ is in a finite set. We may assume 
that $\varepsilon$ is strictly less than the minimal positive 
canonical height of an element in the set containing $\alpha$.

Second, suppose that $[\QQ(\alpha)\colon\QQ] > d 2^k\delta^{-1}$. 
Since $[K(\alpha) \colon K(f^k(\alpha))] \le 2^k$, it follows that
$[\QQ(\alpha) \colon \QQ] \le [K(\alpha) \colon \QQ] \le 2^k 
[K(f^k(\alpha)) \colon \QQ] \le 2^k d [\QQ(f^k(\alpha)) \colon \QQ]$. 
Hence $[\QQ(f^k(\alpha)) \colon \QQ] \ge \delta^{-1}$. Similarly 
$[\QQ(f^{h}(\alpha))\colon\QQ]\ge \delta^{-1}$. 
So $\hat h_f(f^h(\alpha)) <~\delta$ and $\hat h_f(f^{k}(\alpha)) < 
\delta$. On applying~\Cref{thm:weil-to-call-silverman-inequality} to 
$f^h(\alpha)$ and $f^{k}(\alpha)$, it follows that $h(f^h(\alpha)) < 
\hat h_f(f^h(\alpha)) + I + \eta = 2^h\hat h_f(\alpha) + I + \eta$ 
and $h(f^{k}(\alpha)) < 2^k\hat h_f(\alpha) + I +\eta$. Therefore, 
\begin{equation}\label{eqn:infinite-ineq}
	h(f^{k}(\alpha)) + h(f^h(\alpha))
	< (2^h + 2^k)\hat h_f(\alpha) + 2I + 2\eta.
\end{equation}
After substituting~\eqref{eqn:infinite-ineq} into~\eqref{eqn:polya-ineq} 
we conclude that 
\[
	0 \le 
    (2^h + 2^k) \hat h_f(\alpha) - \log 2 + 2I + 2\eta
\]
from which
\[
    \hat h_f(\alpha)
    \ge \frac{\log 2 - 2I - 2\eta}{2^h + 2^k}
    = \frac{\log2 - 2I}{2(2^h + 2^k)}
    > 0.
\]
The proof follows.\qedhere
\end{proof}

\begin{remark}
Instead of a segment, one could use the \textit{hedgehog} generated 
by the zero and pole of the function and estimate its transfinite 
diameter via Dubinin's theorem. The two approaches produce the same 
bound.
\end{remark}

\subsection{Proof of~\Cref{thm:main-thm,thm:main-thm-unr,thm:main-thm-ram}}
Recall that $K$ is a number field unramified above $2$ and fix a 
prime ideal $\p$ of $K$ above $2$. Let $f(T) = T^2 + c \in O_K[T]$ be 
PCF. The cases $c = -2$ and $c = 0$ have been studied respectively by 
Pottmeyer~\cite[Proposition 5.2]{MR3129749} and Amoroso--Zannier~\cite{MR2651944} 
over abelian extensions. Therefore, we can assume $c \notin \{-2, 0\}$.
After replacing $K$ by its Galois closure, we can assume that 
$K / \QQ$ is a Galois extension. Indeed, since $2$ is unramified in 
$K$, it is also unramified in the Galois closure of $K$. 
This follows because the Galois closure is the compositum of the 
conjugated fields $\sigma (K)$ for all $\sigma \in 
\Gal(\overline \QQ / \QQ)$. As each $\sigma$ maps $O_K$ 
isomorphically into $O_{\sigma K}$, the factorization $2 O_K = \p_1 
\cdots \p_g$ gives $2 O_{\sigma(K)} = \sigma(\p_1) \cdots \sigma(\p_g)$, 
from which we conclude that $2$ is unramified in every $\sigma (K)$, 
and consequently is unramified in their compositum.

Let $\alpha \in K^\mathrm{cyc}$ be wandering and fix a root of unity 
$\zeta$ such that $\alpha \in K(\zeta)$. The proof consists of a 
case-by-case analysis.

\textsc{Case 1: Unramified case.}
Suppose first that $2$ is unramified in $K(\zeta)$. The proof in this 
case follows after applying~\Cref{thm:main-technical-thm} with 
$(\sigma, (h, k)) = (\varphi^{m - n}, (dn, dm))$ where $d = [K \colon \QQ]$, 
and $n < m$ are two integers such that $f^{n - 1}(0) = f^{m - 1}(0)$. 
Note that the hypothesis is fulfilled: since $n < m$, if 
$\varphi^{m - n}(f^{dn}(\alpha)) = f^{dm}(\alpha)$, 
then~\Cref{lemma:unramified-equality-case} implies that $\alpha$ is 
preperiodic, a contradiction. This proves~\Cref{thm:main-thm-unr}.

\textsc{Case 2: Ramified case.}
Suppose now that $2$ is ramified in $K(\zeta)$. Choose $\zeta$ so 
that $\alpha \in K(\zeta) \setminus K(\zeta^2)$. Assume that $\alpha 
\in \overline O_\p$ for some prime ideal $\p$ above $2$. Note that if 
$\alpha$ is an algebraic integer, it automatically satisfies this 
hypothesis. 
Consider $\kappa \in \Gal(K(\zeta) / K(\zeta^2))$ as 
in~\Cref{prop:ramified-easy-case}. If $\kappa(f(\alpha)) \neq 
f(\alpha)$, the proof follows again from an application 
of~\Cref{thm:main-technical-thm} with $\sigma = \kappa$, and $h = k = 1$.

If $\kappa(f(\alpha)) = f(\alpha)$, then $f(\alpha) \in K(\zeta^2)$. 
We distinguish two more subcases. 

\textsc{Subcase 2A.}
If $K(\zeta^2) / K$ is unramified above $2$, then \textsc{Case 1} 
applied to the wandering number $f(\alpha)$ implies that there exists 
$\varepsilon > 0$ such that 
\[
    \hat{h}_f(\alpha)
    = \frac{1}{2} \hat{h}_f(f(\alpha))
    \ge \frac{\varepsilon}{2}.
\]

\textsc{Subcase 2B.}
If $K(\zeta^2) / K$ is ramified above $2$, apply again~\Cref{prop:ramified-easy-case} 
to $f(\alpha)$ with $K(\zeta^2)$ in place of $K(\zeta)$ and obtain 
$\rho \in \Gal(K(\zeta^2) / K(\zeta^4))$ with the stated property. If 
$\rho(f^2(\alpha)) \neq f^2(\alpha)$, then the proof follows 
by~\Cref{thm:main-technical-thm} applied to $f(\alpha)$ with the 
field $K(\zeta^2)$ in place of $K(\zeta)$, $\sigma = \rho$, and $h = k = 1$. 

If $K(\zeta^4) / K$ is unramified at $\p$, then the case 
$\rho (f^2(\alpha)) = f^2(\alpha)$ follows as in \textsc{Subcase 2A.}

If $K(\zeta^4) / K$ is ramified above $\p$ and $c \notin \{-2, 0\}$,
\Cref{prop:second-iterate-degeneration} implies that if there exists 
at least one prime ideal $\p$ of $K$ above $2$ such that $\alpha \in 
\overline{O_\p}$, then $\rho(f^2(\alpha)) \neq f^2(\alpha)$. 
Applying~\Cref{thm:main-technical-thm} to $f(\alpha)$ with $h = k = 1$ 
and~$\sigma = \rho$ gives $\hat h_f(f(\alpha)) \ge \varepsilon$. 
Hence $\hat h_f(\alpha) \ge \varepsilon / 2$. This completes the 
proofs of~\Cref{thm:main-thm,thm:main-thm-ram}.

\defbibheading{mybibheading}{
  \section*{References}
}

\printbibliography[heading = mybibheading]
\end{document}